\documentclass[a4paper,12pt]{article}
\usepackage{amsfonts}
\usepackage{amstext}
\usepackage{amsthm}
\usepackage{amsmath}
\usepackage{amssymb}
\usepackage{latexsym}
\usepackage{graphicx}
\usepackage{caption}
\usepackage[hidelinks]{hyperref}
\usepackage{enumitem}
\usepackage[utf8]{inputenc}
\newcommand{\R}{\Bbb{R}}
\newcommand{\N}{\Bbb{N}}

\newcommand{\s}{\Bbb{S}}

\newtheorem{teor}{Theorem}[section]
\newtheorem{propo}{Proposition}[section]
\newtheorem{defi}{Definition}[section]

\newtheorem{rem}{Remark}[section]
\newtheorem{ex}{Example}[section]

\newcommand{\n}{\noindent}
\newcommand{\vs}{\vspace}

\DeclareMathOperator{\spn}{span}
\DeclareMathOperator{\interior}{int}

\begin{document}

\title{Sharp isoanisotropic estimates for fundamental frequencies of membranes and connections with shapes
}

\author{\textbf{Raul Fernandes Horta \footnote{\textit{E-mail addresses}: raul.fernandes.horta@gmail.com
 (R. F. Horta)}}\\ {\small\it Departamento de Matem\'{a}tica,
Universidade Federal de Minas Gerais,}\\ {\small\it Caixa Postal
702, 30123-970, Belo Horizonte, MG, Brazil}\\
\textbf{Marcos Montenegro \footnote{\textit{E-mail addresses}:
montene@mat.ufmg.br (M. Montenegro)}}\\ {\small\it Departamento de Matem\'{a}tica,
Universidade Federal de Minas Gerais,}\\ {\small\it Caixa Postal
702, 30123-970, Belo Horizonte, MG, Brazil}}

\date{}

\maketitle

\markboth{abstract}{abstract}

\hrule \vspace{0,2cm}

\n {\bf Abstract}

The underlying motivation of the present work lies on a cornerstone question in spectral optimization that consists of determining sharp lower and upper uniform estimates for fundamental frequencies of a set of uniformly elliptic operators on a fixed membrane. We solve completely the problem in the plane for the general class of anisotropic operators in divergence form generated by arbitrary norms, which also includes the computation of optimal constants and the characterization of corresponding anisotropic extremizers (if they exist). Our approach is based on an isoanisotropic optimization formulation which, in turn, demands to be addressed within the broader environment of nonnegative, convex and $1$-homogeneous anisotropies. A fine and detailed analysis of least energy levels associated to anisotropies with maximum degeneracy leads to a central connection between shapes and fundamental frequencies of rather degenerate elliptic operators. Such a linking also permits to establish that the supremum of anisotropic fundamental frequencies over all fixed-area membranes is infinite for any nonzero anisotropy. As a by-product, the well-known maximization conjecture for fundamental frequencies of the $p$-Laplace operator is proved for any $p \neq 2$.


\pagebreak

\vspace{0.5cm}
\hrule\vspace{0.2cm}

\tableofcontents

\pagebreak

\section{Introduction}

\subsection{Overview}\label{ssec: Overview}

A very old central line in spectral optimization consists in the study of several problems relating shape of domains and/or elliptic operators to the corresponding spectra under different boundary conditions. A landmark chapter within this field revolves around the celebrated Faber-Krahn isoperimetric inequality conjectured in the physical context in 1887 by Rayleigh in his book \cite{R}:

\begin{center}
{\it Among all drums with membranes of fixed area,\\ the circular drum produces the lowest frequency of sound.}
\end{center}
This is one among several interesting physical situations that can be mathematically modeled as follows.

For a bounded domain $\Omega$ in $\R^n$, consider the first Dirichlet eigenvalue $\lambda_1(\Omega)$ associated to the Laplace operator

\[
-\Delta := - \sum_{i = 1}^n \frac{\partial^2}{\partial {x_i}^2},
\]
also called {\it fundamental frequency} of the domain $\Omega$, often called {\it membrane}.

The referred inequality states that

\begin{equation} \label{FK}
\lambda_1(\Omega) \geq \lambda_1(B) \tag{FK}
\end{equation}
for every $\Omega$ with $\vert \Omega \vert = \vert B \vert$, where $B$ denotes the unit $n$-Euclidean ball and $\vert \cdot \vert$ stands for the Lebesgue measure of a measurable subset of $\R^n$. Moreover, equality holds in \eqref{FK} if and only if $\Omega$ is equal to $B$, up to a translation and a set of zero capacity. As the name itself suggests, the inequality was independently proved by Faber \cite{F} and Krahn \cite{Kr} using the variational characterization of eigenvalues and symmetrization techniques.

Optimizing shapes within spectral optimization linked to second-order elliptic operators is, at the same time, fascinating and generally difficult, which has aroused great interest in the mathematical community. A number of counterparts of \eqref{FK}, all motivated by physical applications, has been established over the past three decades. Regarding the first nonzero eigenvalue of the Laplace operator under different boundary conditions, we select some quite famous inequalities:

\begin{itemize}
\item[$\bullet$] Bossel-Daners inequality \cite{B, BD, D, DK} (Robin condition) (see also \cite{DPG, DPP});
\item[$\bullet$] Szegö-Weinberger inequality \cite{S, W} (Neumann condition) (see also \cite{BH});
\item[$\bullet$] Brock-Weinstock inequality \cite{Br, We} (Steklov condition).
\end{itemize}
Faber-Krahn type inequalities have also been proved for other elliptic equations/operators as for example the Schrödinger equation \cite{DCH}, the $p$-Laplace operator \cite{AFT}, anisotropic operators \cite{BFK}, the fractional Laplace operator \cite{BLP} and local-nonlocal type operators \cite{BDVV}. Other related issues such as quantitative forms and stability of geometric inequalities with respect to domains have been extensively discussed for the spectrum of elliptic operators. Among a diverse literature, we refer to \cite{BDPV, BFNT, FZ} for some quantitative variants and \cite{BCV, FMP, HN, M} for results on stability. For an overview about miscellaneous problems, improvements and open questions on several related topics, we recommend the excellent books \cite{BB, H1, H3} and the collection of contributions \cite{H2}.

When we drop out the constraint $\vert \Omega \vert = \vert B \vert$ in \eqref{FK}, thanks to the scaling property of $\lambda_1(\Omega)$ with respect to $\Omega$, the Faber-Krahn inequality takes the form

\begin{equation} \label{FK1}
\lambda_1(\Omega) \geq \vert B \vert^{2/n} \lambda_1(B) \, \vert \Omega \vert^{-2/n},
\end{equation}
which is also very relevant from the analytical viewpoint once it provides a sharp lower estimate for the fundamental frequency $\lambda_1(\Omega)$ in terms of the measure of $\Omega$ with universal optimal constant $\vert B \vert^{2/n} \lambda_1(B)$. In this sense, another direction that has gained a lot of traction is the study of sharp uniform estimates of fundamental frequencies associated with a given set of elliptic operators defined in a fixed membrane $\Omega$. In more precise terms, let $\Omega$ be a bounded domain in $\R^n$ and consider a broad set ${\cal E}(\Omega)$ of second-order uniformly elliptic operators on $\Omega$ so that each operator ${\cal L} \in {\cal E}(\Omega)$ admits at least a nonzero fundamental frequency $\lambda_1({\cal L})$ under some a priori fixed boundary condition (e.g. Dirichlet, Neumann, Robin). A counterpart of \eqref{FK1} associated with the set ${\cal E}(\Omega)$ relies on a suitable way of ``measuring" operators ${\cal L} \in {\cal E}(\Omega)$ by some quantity denoted say by ${\cal M}({\cal L})$.

Two keystone questions in optimization of ``shapes" of operators in ${\cal E}(\Omega)$ are:

\begin{enumerate}[label=(\Alph*)]
\item\label{itemA} Are there any explicit optimal constants $\Lambda^{\min}(\Omega)$ and $\Lambda^{\max}(\Omega)$ such that

\begin{equation} \label{SSE}
\Lambda^{\min}(\Omega)\, {\cal M}({\cal L}) \leq \lambda_1({\cal L}) \leq \Lambda^{\max}(\Omega)\, {\cal M}({\cal L})
\end{equation}
for every operator ${\cal L} \in {\cal E}(\Omega)$?

\item\label{itemB} Are there operators ${\cal L} \in {\cal E}(\Omega)$ yielding any equality in \eqref{SSE}? Can such operators be characterized?
\end{enumerate}
Logically, the measure ${\cal M}({\cal L})$ is expected to be introduced naturally so that the lower constant $\Lambda^{\min}(\Omega)$ is positive and the upper $\Lambda^{\max}(\Omega)$ is finite. Both inequalities in \eqref{SSE} provide sharp uniform bounds for nonzero fundamental frequencies $\lambda_1({\cal L})$ over all operators in ${\cal E}(\Omega)$. Notice also that the first of them can be seen as an analytical parallel of the Faber-Krahn geometric inequality. An operator ${\cal L} \in {\cal E}(\Omega)$ that solves \ref{itemB} is called an extremizer of \eqref{SSE}.

Within that spirit, Essen \cite{E} started in the 80s the study in dimension $n = 1$ of lower and upper uniform estimates for Dirichlet eigenvalues $\lambda_k(V)$ of the Schrödinger operator

\[
{\cal L}_S u = -\hbar \Delta u + V(x) u\ \ {\rm in}\ \Omega,
\]
where $\hbar$ denotes the Planck's constant and $V$ is a potential function satisfying the constraints

\[
\Vert V \Vert_{L^p(\Omega)} \leq \kappa\ \ {\rm and}\ \ \kappa_1 \leq V \leq \kappa_2\ \ {\rm in}\ \Omega
\]
for fixed constants $\kappa$, $\kappa_1$ and $\kappa_2$. Many advances have been made over decades toward non-sharp and sharp uniform estimates in any dimension. We refer for example to \cite{CMZ, EK, E, K, KS, N, T, ZW} for the one-dimensional case and \cite{AM, BBV, BGRV, CGIKO, Eg, H, MRR} and Chapter 9 of \cite{H2} for higher dimensions.

In a similar line, sharp uniform estimates have also been obtained for Dirichlet eigenvalues $\lambda_k(\sigma)$ of the conductivity operator

\[
 {\cal L}_C u = -{\rm div}(\sigma(x) \nabla u)\ \ {\rm in}\ \Omega
\]
for coefficients $\sigma$ normalized simultaneously by the above $L^p$ and uniform restrictions with constants $\kappa_1, \kappa_2 > 0$, see \cite{Ba, CL, EKo, Tr} for some developments.

On the other hand, relatively little is known about spectral optimization dealing with a broader set ${\cal E}(\Omega)$ of linear and nonlinear elliptic operators in any dimension $n$. We develop a comprehensive theory in dimension $n = 2$ for the general class of homogeneous anisotropic elliptic operators. The questions \ref{itemA} and \ref{itemB} are completely solved for fairly general membranes $\Omega$ by means of an appropriate optimization framework. Part of the solution requires a fine analysis of fundamental frequencies associated to operators with maximum degeneracy. The strength of our results allows us to give a simple proof of the well-known maximization conjecture, open for any $p \neq 2$, which states that the supremum of all fundamental frequencies of the $p$-Laplacian over fixed area membranes is infinite. In this new context, the first inequality in \eqref{SSE} represents the isoanisotropic counterpart of the Faber-Krahn isoperimetric inequality \eqref{FK1} in the plane. The results obtained here are quite complete in the sense that all optimal constants and extremal operators are exhibit explicitly.

This paper is the first that aims to establish a satisfactory optimization theory in dimension 2 within the proposed program \ref{itemA} and \ref{itemB}. The picture changes drastically in higher dimensions and new shape phenomena arise, we refer to the work \cite{HM} that is currently underway.

\subsection{2D anisotropies}

In this subsection we will present the appropriate anisotropic environment for the development of the program \ref{itemA} and \ref{itemB}, including the characterization of degenerate prototypes, which will be useful later.

Consider the following class of functions, called here $2D$ {\it anisotropies}:

\[
{\cal H} = \left\{H \colon \R^2 \to \R \colon H \ \text{is nonnegative, convex and $1$-homogeneous}\right\},
\]
where the space $\R^{2}$ is assumed to be endowed with the usual Euclidean norm denoted by $|\cdot|$.

Recall that a function $H$ is said to be convex if satisfies

\[
H(t(x,y) + (1-t)(z,w)) \leq tH(x,y) + (1-t)H(z,w)
\]
for all $t \in [0,1]$ and $(x,y), (z,w) \in \R^2$, and is $1$-homogeneous if

\[
H(t(x,y)) = \vert t \vert H(x,y)
\]
for all $t \in \R$ and $(x,y) \in \R^2$.

Some typical examples of functions in ${\cal H}$ are:

\begin{ex} \label{ex1}

Seminorms and Norms:

\n $\blacktriangleright$ $H(x,y) = \vert x \vert$ and $H(x,y) = \vert y \vert$.\\ More generally:\\
\indent $H(x,y) = \vert \cos{\theta}\, x + \sin{\theta}\, y\, \vert$ for an angle $\theta \in [0,\pi]$.

\n $\blacktriangleright$ $H(x,y) = \vert (x,y) \vert = (x^2 + y^2)^{1/2}$.\\ More generally:\\
\indent $H(x,y) = \left( \vert x \vert^p + \vert y \vert^p \right)^{1/p}$ for a parameter $p \geq 1$,\\
\indent $H(x,y) = \vert (x,y)^T A\, (x,y) \vert^{1/2}$ for an invertible $2 \times 2$ matrix $A$,\\
\indent $H(x,y) = \Vert (x,y) \Vert$ for an arbitrary norm $\Vert \cdot \Vert$.
\end{ex}

Note that ${\cal H}$ can be seen as the set of all seminorms on $\R^2$. However, a convenient and clever way to view ${\cal H}$ is as a closed subset of the Banach space $X$ of all $1$-homogeneous continuous functions equipped with the norm

\[
\lVert H \rVert = \max\{ H(x,y):\ |(x,y)| = 1\}.
\]
As will become clear later, this closure property explains why we choose ${\cal H}$, rather than the smaller set of all positive convex anisotropies, in order to successfully carry out the spectral optimization program \ref{itemA} and \ref{itemB}.

According to the examples above, it is natural to decompose ${\cal H}$ into the sets of positive and degenerate $2D$ anisotropies to be denoted respectively by

\begin{eqnarray*}
&& {\cal H}^P:= \left\{H \in {\cal H}:\ H(x,y) > 0,\ \forall (x,y) \neq (0,0)\right\},\\
&& {\cal H}^D:= \left\{H \in {\cal H}:\ H(x,y) = 0\ \text{for some}\ (x,y) \neq (0,0)\right\}.
\end{eqnarray*}
The set ${\cal H}^P$ is open but is not closed in $X$ and ${\cal H}^D$ is closed in $X$. Besides, it is clear that ${\cal H}^P$ is the set of all norms on $\R^2$.

The set ${\cal H}^D$ is easily characterizable as:

\begin{propo} \label{P1}
The set of all degenerate $2D$ anisotropies is characterized as

\[
{\cal H}^D= \left\{H(x,y) = c \vert \cos{\theta}\, x + \sin{\theta}\, y\, \vert:\ c \geq 0,\ \theta \in [0,\pi]\right\}.
\]
\end{propo}

\begin{proof} Let $H \in {\cal H}^D \setminus \{0\}$. One knows that there is $\theta \in [0,\pi]$ such that $H(-\sin{\theta},\cos{\theta}) = 0$. Let $T_\theta$ be the rotation matrix associated to the angle $\frac{\pi}{2}+\theta$, then

\begin{align*}
& T_\theta(1,0) = \left(-\sin{\theta},\cos{\theta}\right), \\
& T_\theta(0,1) = \left(-\cos{\theta},-\sin{\theta}\right).
\end{align*}
Now define $H_\theta(x,y) := H(T_\theta(x,y))$. It is clear that $H_\theta(1,0) = 0$. Thus, for any $(x,y) \in \R^2$, we have

\begin{align*}
   & H_\theta(x,y) \leq H_\theta(x,0) + H_\theta(0,y) = |x|H_\theta(1,0) + H_\theta(0,y) = H_\theta(0,y), \\
   & H_\theta(0,y) \leq H_\theta(x,y) + H_\theta(-x,0) = H_\theta(x,y) + |x|H_\theta(1,0) = H_\theta(x,y). \\
\end{align*}
Then, it follows that $H_\theta(x,y) = H_\theta(0,y)$, in other words, we have

\[
H_\theta(x,y) = |y|H_\theta(0,1) = H_\theta(0,1)\, |\langle (x,y) , (0,1) \rangle|.
\]
From this equality, we derive

\begin{align*}
H(x,y) &= H(T_\theta(T_\theta^{t}(x,y))) = H_\theta(T_\theta^{t}(x,y)) = H_\theta(0,1)\, |\langle T_\theta^{t}(x,y) , (0,1) \rangle| \\
&= H_\theta(0,1)\, |\langle (x,y), T_\theta(0,1)\rangle| = H_\theta(0,1)\, |\langle (x,y), \left(-\cos{\theta},-\sin{\theta}\right)\rangle| \\
& = H_\theta(0,1)\, |-x\cos{\theta}-y\sin{\theta}| = H_\theta(0,1)\, |x\cos{\theta}+y\sin{\theta}|.
\end{align*}
Therefore, the conclusion follows with $c = H_\theta(0,1) > 0$.
\end{proof}

\subsection{Anisotropic least levels versus fundamental frequencies}\label{ssec: Anisotropic least levels versus fundamental frequencies}
We now raise some natural and important questions dealing with anisotropic energy levels as part of the path that will lead us to the most relevant questions, as well as their solutions. Some partial answers will still be provided in this subsection, while others will be given later.

Given a real number $p>1$, a membrane $\Omega \subset \R^2$ and an $2D$ anisotropy $H \in {\cal H}$, we introduce the anisotropic $L^p$ energy associated to $H$ as the functional ${\cal E}_{p,H} \colon W^{1,p}_{0}(\Omega) \to \R$ defined by

\[
{\cal E}_{p,H}(u) := \iint_{\Omega} H^{p}(\nabla u) \, dA,
\]
where $\nabla u$ stands for the weak gradient of $u$ and $W^{1,p}_{0}(\Omega)$ denotes the completion of compactly supported smooth functions in $\Omega$ with respect to the norm

\[
\lVert u \rVert_{W^{1,p}_{0}(\Omega)} = \left(\iint_{\Omega} | \nabla u|^{p} \, dA\right)^{\frac{1}{p}}.
\]
The anisotropic least energy level associated to ${\cal E}_{p,H}$ on the unit sphere in $L^p(\Omega)$ is defined as

\begin{equation}\label{frequency}
\lambda_{1,p}^{H}(\Omega) := \inf\left\{{\cal E}_{p,H}(u) \colon\ u \in W^{1,p}_0(\Omega), \ \lVert u \rVert_{p} = 1\right\},
\end{equation}
where
\[
\lVert u \rVert_p = \left(\iint_{\Omega} |u|^{p} \, dA\right)^{\frac{1}{p}}.
\]

Before we begin discussing issues surrounding $\lambda_{1,p}^{H}(\Omega)$, we have gathered some basic properties that follow directly from the definition and that will be used throughout the work.

\begin{enumerate}[label=(P\arabic*)]
    \item\label{PropertyP1} If $H \in {\cal H}$ and $\alpha \in [0,+\infty)$ then $\alpha H \in {\cal H}$ and $\lambda_{1,p}^{\alpha H}(\Omega) = \alpha^p \lambda_{1,p}^{H}(\Omega)$ for any membrane $\Omega$;

    \item\label{PropertyP2} If $H \in {\cal H}$ and $\Omega_{1}$ and $\Omega_{2}$ are membranes such that $\Omega_{1} \subseteq \Omega_{2}$ then $\lambda_{1,p}^{H}(\Omega_{1}) \geq \lambda_{1,p}^{H}(\Omega_{2})$;

    \item\label{PropertyP3} If $G$ and $H$ are functions in ${\cal H}$ such that $G \leq H$, then $\lambda_{1,p}^{G}(\Omega) \leq \lambda_{1,p}^{H}(\Omega)$ for any membrane $\Omega$.
\end{enumerate}

Given a membrane $\Omega$, an $2D$ anisotropy $H \in {\cal H}$ (resp. ${\cal H}^P, {\cal H}^D$), a rotation matrix $A$ and a function $u \in W^{1,p}_{0}(\Omega)$, we denote $\Omega_{A} = A(\Omega)$, $H_{A} = H \circ A$ and $u_A$ the function given by $u_{A}(x,y) = u(A^T(x,y))$. It is clear that $H_{A} \in {\cal H}$ (resp. ${\cal H}^P, {\cal H}^D$) and $u_A \in W^{1,p}_{0}(\Omega_A)$.

We also consider a fourth useful property:

\begin{enumerate}[label=(P\arabic*),start=4]
    \item\label{PropertyP4} If $H \in {\cal H}$ then $\lambda_{1,p}^{H}(\Omega) = \lambda_{1,p}^{H_A}(\Omega_A)$ for any rotation matrix $A$. Moreover, $u \in W^{1,p}_{0}(\Omega)$ is a minimizer of $\lambda_{1,p}^{H}(\Omega)$ if and only if $u_A \in W^{1,p}_{0}(\Omega_A)$ is a minimizer of $\lambda_{1,p}^{H_A}(\Omega_A)$.
\end{enumerate}
The latter follows readily from the following relations, after using the change of variable $(z,w) = A(x,y)$ for $(x,y) \in \Omega$,

\[
\iint_{\Omega_A} H_{A}^{p}(\nabla u_{A}(z,w))\, dA = \iint_{A(\Omega)} H^p(A A^T \nabla u (A^T (z,w))\, dA = \iint_{\Omega} H^p(\nabla u(x,y))\, dA
\]
and

\[
\iint_{\Omega_{A}} |u_{A}(z,w)|^p\, dA = \iint_{\Omega} |u(x,y)|^p\, dA.
\]

Back to our main focus, a first question related to $\lambda_{1,p}^{H}(\Omega)$ already arises:

\begin{center}
\it{For which $2D$ anisotropies $H \in {\cal H}$ is $\lambda_{1,p}^{H}(\Omega)$ nonzero?}
\end{center}

The complete answer to this query is given in a simple way in the following result:

\begin{propo} \label{P.1}
Let $\Omega$ be any membrane and let $H \in {\cal H}$. The statements are equivalent:

\begin{itemize}
\item[(a)] The level $\lambda_{1,p}^{H}(\Omega)$ is positive;
\item[(b)] The anisotropy $H$ is nonzero at some point.
\end{itemize}
\end{propo}

\begin{proof}
Since clearly $\lambda_{1,p}^{H}(\Omega) = 0$ for $H = 0$, it suffices to prove that (b) implies (a). Assume then that $H$ is nonzero at some point. We analyze two possibilities. For $H \in {\cal H}^P$, by continuity and homogeneity, there is a constant $c > 0$ such that $H(x,y) \geq c \vert (x,y) \vert \geq c \vert y \vert$ for every $(x,y) \in \R^2$. For $H \in {\cal H}^D \setminus \{0\}$, from Proposition \ref{P1}, we know that $H(x,y) = c \vert \cos{\theta}\, x + \sin{\theta}\, y\, \vert$ for some constant $c >0$. Thus, if $A$ is the rotation matrix of angle $\theta - \frac{\pi}{2}$, then $H_A(x,y) = c \vert y \vert$. In any case, there is an orthogonal matrix $A$ (possibly the identity matrix) such that $H_A(x,y) \geq c \vert y \vert$ where $c > 0$. From Properties \ref{PropertyP1}, \ref{PropertyP3} and \ref{PropertyP4}, it suffices to consider $H(x,y) = \vert y \vert$. Furthermore, if $R = (a,b) \times (c,d)$ is a convenient rectangle so that $\Omega \subset R$, by Property \ref{PropertyP2}, it suffices to show that $\lambda_{1,p}^{H}(R) > 0$.

On the other hand, by the one-dimensional Poincaré inequality, we have

\begin{eqnarray*}
\iint_{R} H^p(\nabla u(x,y))\, dA &=& \int_a^b \int_c^d \vert \frac{\partial u}{\partial y}(x,y) \vert^p\, dy\, dx \geq \int_a^b \lambda_{1,p}(c,d) \int_c^d \vert u(x,y) \vert^p\, dy\, dx  \\
&=& \lambda_{1,p}(c,d) \iint_{R} \vert u(x,y) \vert^p\, dA
\end{eqnarray*}
for every $u \in C^1_0(R)$, so that by density $\lambda_{1,p}^{H}(R) \geq \lambda_{1,p}(c,d)  > 0$.
\end{proof}

A second natural question is

\begin{center}
\it{For which $2D$ anisotropies $H \in {\cal H}$ is the level $\lambda_{1,p}^{H}(\Omega)$ attained\\ by some $L^p$-normalized function $u_p \in W^{1,p}_{0}(\Omega)$?}
\end{center}
Admitting for a moment that such a minimizer $u_p$ exists, the next step is to know whether it satisfies any Euler-Lagrange equation. This is the case for instance when $H \in C^1(\R^2 \setminus \{(0,0)\})$. In this case, the anisotropic energy ${\cal E}_{p,H}$ is Gateaux differentiable in $W^{1,p}_0(\Omega)$ and, as can easily be checked, the function $u_p$ is a weak solution of the equation

\begin{equation} \label{EP}
-\Delta_p^H u = \lambda_{1,p}^H(\Omega) |u|^{p-2}u \ \text{in} \ \Omega,
\end{equation}
where $-\Delta_{p}^{H}$ denotes the quasilinear elliptic operator in divergence form

\[
-\Delta_p^H u := -{\rm div}\left(H^{p-1}\left(\nabla u\right)\nabla H\left(\nabla u\right)\right),
\]
which will henceforth be called {\it $H$-anisotropic $p$-Laplace operator}. A function $u_p \in W^{1,p}_0(\Omega)$ is said to be solution of \eqref{EP} in the weak sense if
\[
\iint_{\Omega} H^{p-1}\left(\nabla u_p\right)\nabla H\left(\nabla u_p\right)\cdot\nabla\varphi \, dA = \lambda_{1,p}^{H}(\Omega)\iint_{\Omega}|u_p|^{p-2}u_p\varphi \, dA
\]
for every test function $\varphi \in W^{1,p}_0(\Omega)$.

At this point it should be remarked that the gradient $\nabla H$ of $2D$ anisotropies exists almost everywhere and is bounded in $\R^2$, so that the integral on the left-hand side above, if well defined, is always finite. So, in general, the level $\lambda_{1,p}^H(\Omega)$ is called {\it $H$-anisotropic fundamental $p$-frequency} (or just {\it fundamental frequency}) of the membrane $\Omega$ if \eqref{EP} admits a nontrivial weak solution $u_p \in W_0^{1,p}(\Omega)$ which is called {\it membrane eigenfunction} associated to $\lambda_{1,p}^H(\Omega)$.

Consequently, as a continuation of the second question, we are right away led to:

\begin{center}
\it{For which $2D$ anisotropies $H \in {\cal H}$ is the level $\lambda_{1,p}^{H}(\Omega)$ a fundamental frequency?}
\end{center}

As it is well known, the second question is affirmative for any $H \in {\cal H}^P$ and the third under the additional assumption that $H$ is of $C^1$ class in $\R^2$ outside the origin (e.g. \cite{BFK}). The novelty here is that this latter is actually necessary in the non-degenerate case. In precise terms, we have:

\begin{teor} \label{T6}
Let $H \in {\cal H}^P$. The statements are equivalent:

\begin{itemize}
\item[(a)] The level $\lambda_{1,p}^{H}(\Omega)$ is a fundamental frequency for any membrane $\Omega$;
\item[(b)] The anisotropy $H$ is of $C^1$ class in $\R^2 \setminus \{(0,0)\}$.
\end{itemize}
\end{teor}

\begin{proof} As pointed out above, it suffices to prove that $(a)$ implies $(b)$. Assume by contradiction that $(b)$ is false. Consider the convex body ($H \in {\cal H}^P$)

\[
D_H = \left\{ (x,y) \in \R^2:\ H(x,y) \leq 1\right\}
\]
and its support function

\[
H^\circ(x,y) = \max\{xz + yw:\ (z,w) \in D_H\},
\]
which it is also a norm on $\R^2$. The polar convex body of $D_H$ is defined by

\[
D^\circ_H = \left\{ (x,y) \in \R^2:\ H^\circ(x,y) \leq 1\right\}.
\]
Set $\Omega = D_{H}^{\circ}$ and take a membrane eigenfunction $u_p \in W^{1,p}_0(\Omega)$ associated to $\lambda_{1,p}^{H}(\Omega)$. By Theorem 3.1 of \cite{AFTL}, we have

\begin{equation}\label{ineqconvexsym}
\lambda_{1,p}^{H}(\Omega) = \iint_{\Omega} H^{p}(\nabla u_p)\, dA \geq \iint_{\Omega} H^{p}(\nabla u_p^{\star})\, dA
\end{equation}
and $\Vert u_p^{\star} \Vert_{p} = \Vert u_p \Vert_{p} = 1$, where $u_p^{\star}$ denotes the convex rearrangement of $u_p$ with respect to $H$, which also belongs to $W^{1,p}_0(\Omega)$. Then, \eqref{ineqconvexsym} is actually an equality and therefore $u_p = u_p^{\star}$. Hence, there is a continuous function $l_p$ differentiable almost everywhere in $\R$ that satisfies

\[
u_p(x,y) = l_p(H^{\circ}(x,y)).
\]
Since $H$ is Lipschitz in $\R^2$ note that $H \in C^{1}(\R^2 \setminus \{(0,0)\})$ if and only if $H$ is Gateaux differentiable at every point in $\R^2 \setminus \{(0,0)\}$. Suppose by contradiction that $H$ is not Gateaux differentiable at some point $(z,w) \neq (0,0)$. Consider the subdifferential $\partial H(z,w)$ of $H$ at the same point, that is characterized as

\begin{equation}\label{normsubdiff}
\partial H(z,w) = \left\{(x,y) \in \R^2 \colon\ z x + w y = H(z,w) \ \text{and} \ H^{\circ}(x,y) = 1 \right\}.
\end{equation}
From the Gateaux non-differentiability of $H$ at $(z,w)$, the set $\partial H(z,w)$ contains at least two linearly independent vectors and, by convexity, it contains at least a line segment that does not pass through the origin. Consequently, the conical set

\[
{\cal C}_H(z,w) = \left\{\alpha (x,y) \in \R^2:\  \forall (x,y) \in \partial H(z,w), \ \forall \alpha \in \R\right\}
\]
has positive measure and thus $H^{\circ}$ is differentiable in almost every point in it.

On the other hand, from \eqref{normsubdiff}, we have

\[
{\cal C}_H(z,w) = \left\{(x,y) \in \R^2 \colon\ z x + w y = H(z,w) H^{\circ}(x,y)\right\}.
\]
Choose now a point $(x_0,y_0) \in {\cal C}_H(z,w)$ such that $H^{\circ}$ is differentiable in $(x_0,y_0)$. Notice that this is a maximum point of the function $F(x,y) = z x + w y$ restricted to the set $H^{\circ}(x,y) = H(x_0,y_0)$. So, by Lagrange multipliers, we have that there is $\lambda \in \R$ such that

\[
\nabla H^{\circ}(x_0,y_0) = \lambda \nabla F(x_0,y_0) = \lambda (z,w),
\]
Therefore, we derive

\[
\nabla u_p (x_0,y_0) = l_p^\prime \left(H^{\circ}(x_0,y_0) \right) \nabla H^{\circ}(x_0,y_0) \in \spn{(z,w)} \ \text{a.e in} \ \Omega \cap {\cal C}_H(z,w).
\]
Since $|\Omega \cap {\cal C}_H(z,w)| > 0$, then $\nabla H(\nabla u_p)$ is not well-defined almost everywhere in $\Omega$. Consequently, there is no membrane eigenfunction associated to $\lambda_{1,p}^{H}(\Omega)$ for $\Omega = D_{H}^{\circ}$.
\end{proof}
\n Finally, regarding degenerate $2D$ anisotropies, the second and third questions are deeper and part of this work aims to provide complete answers in this case as well. Among the main statements (of Subsection \ref{ssec: Main statements}), we furnish, for each anisotropy $H \in {\cal H}^D \setminus \{0\}$, the characterization of all shapes $\Omega$ (depending on $H$) where $\lambda_{1,p}^{H}(\Omega)$ are fundamental frequencies. Illustrations of different shapes $\Omega$ will also be exhibited in which $\lambda_{1,p}^{H}(\Omega)$ is or is not a fundamental frequency, including more fine-grained information about the exact number (multiplicity) of degenerate anisotropies $H$ such that $\lambda_{1,p}^{H}(\Omega)$ is a fundamental frequency. More specifically, given any integer number $m \geq 1$, it will be constructed an example of membrane $\Omega$ such that $\lambda_{1,p}^{H}(\Omega)$ is a fundamental frequency for exactly $m$ anisotropies $H \in {\cal H}^D$ normalized by $\Vert H \Vert = 1$.

\subsection{Isoanisotropic problems on a given membrane}

We next introduce anisotropic spectral optimization problems that have as a backdrop the anisotropic Faber-Krahn inequality established by Belloni, Ferone and Kawohl \cite{BFK} for anisotropies $H \in {\cal H}^P$ of $C^1$ class in $\R^2 \setminus \{(0,0)\}$. It states that, for any membrane $\Omega$ with $\vert \Omega \vert = \vert D_H^\circ \vert$,

\begin{equation} \label{AFK}
\lambda^H_{1,p}(\Omega) \geq \lambda^H_{1,p}(D_H^\circ),
\end{equation}
where $D_H^\circ$ is the polar convex body of $D_H = \{(x,y) \in \R^2: H(x,y) \leq 1\}$ as defined in the previous subsection. Moreover, equality holds in \eqref{AFK} if and only if $\Omega$ is equal to $D_H^\circ$, module a translation and a set of zero capacity. Observe that \eqref{AFK} is an inequality of isoperimetric nature that solves and classifies the shape minimization problem

\[
\min \{\lambda^H_{1,p}(\Omega):\, \forall\, \text{membrane\ } \Omega \subset \R^2\ {\rm with}\ \vert \Omega \vert = 1\}.
\]
As for the corresponding maximization problem, we will establish that (see Theorem \ref{T2.1})

\[
\sup \{\lambda^H_{1,p}(\Omega):\, \forall\, \text{membrane\ } \Omega \subset \R^2\ {\rm with}\ \vert \Omega \vert = 1\} = \infty.
\]
When $\Omega$ is fixed and $H$ varies freely in ${\cal H}$, it is natural to ask about the anisotropic parallel of shape optimization and clearly this depends on the kind of ``measurement" adopted for elements $H \in {\cal H}$. There are various possible choices generally related to different ways of measuring the convex set $D_H$. Here we opted for the {\it inradius} measure of this set which just coincides with the inverse of $\Vert H \Vert$ defined in the overview section. For this reason, in what follows, we must elect the notation $\Vert H \Vert$ rather than the {\it inradius} of $D_H$. We recall that the {\it inradius} of a given set in $\R^2$ is the radius of the largest open disk contained in it.

In order to introduce the isoanisotropic optimization problems to fundamental frequencies $\lambda_{1,p}^H(\Omega)$ of a fixed membrane $\Omega$, we consider the unit anisotropic sphere $\s({\cal H}) = \{H \in {\cal H}: \Vert H \Vert = 1\}$.

It basically consists of two types:\\

\vs{-0.4cm}

\begin{center}

{\bf Isoanisotropic minimization:}

\vs{0.2cm}

{\it Among all $2D$ anisotropies $H \in \s({\cal H})$, which of them (if any) produces\\ the lowest fundamental frequency $\lambda_{1,p}^H(\Omega)$?}
\end{center}

\begin{center}

{\bf Isoanisotropic maximization:}

\vs{0.2cm}

{\it Among all $2D$ anisotropies $H \in \s({\cal H})$, which of them (if any) produces\\ the highest fundamental frequency $\lambda_{1,p}^H(\Omega)$?}
\end{center}
In other words, both problems consist in characterizing all anisotropic extremizers of $\lambda_{1,p}^H(\Omega)$ on $\s({\cal H})$.

Consider the optimal anisotropic constants corresponding to each of these problems

\[
\lambda_{1,p}^{\min}(\Omega):= \inf_{H \in \s({\cal H})} \lambda_{1,p}^H(\Omega)\ \ {\rm and}\ \ \lambda_{1,p}^{\max}(\Omega):= \sup_{H \in \s({\cal H})} \lambda_{1,p}^H(\Omega)\, .
\]
It is clear that $0 \leq \lambda_{1,p}^{\min}(\Omega) \leq \lambda_{1,p}^{\max}(\Omega)$. Assuming for now that $\lambda_{1,p}^{\min}(\Omega) > 0$ and $\lambda_{1,p}^{\max}(\Omega)$ is finite, we find two non-vacuum sharp isoanisotropic inequalities:

\begin{center}
{\bf Sharp lower anisotropic inequality:}
\end{center}

\vs{-0.6cm}

\begin{equation} \label{LII}
\lambda_{1,p}^H(\Omega) \geq \lambda_{1,p}^{\min}(\Omega)\ \ \text{for every\ } H \in \s({\cal H}) \tag{LAI}
\end{equation}

\vs{0.1cm}

\begin{center}
{\bf Sharp upper anisotropic inequality:}
\end{center}

\vs{-0.6cm}

\begin{equation} \label{UII}
\lambda_{1,p}^H(\Omega) \leq \lambda_{1,p}^{\max}(\Omega)\ \ \text{for every\ } H \in \s({\cal H}) \tag{UAI}
\end{equation}
The sharp inequality \eqref{LII} is the isoanisotropic counterpart of the Faber-Krahn isoperimetric inequality \eqref{FK} in the plane. The strategy of solution of two above optimization problems will be based on the study of anisotropic extremizers $H \in \s({\cal H})$ for both inequalities \eqref{LII} and \eqref{UII}. More specifically, it will be established that (see Subsection \ref{ssec: Main statements} for precise statements):

\begin{itemize}
\item[(I)\ ] $\lambda_{1,p}^{\min}(\Omega)$ is always positive;

\item[(II)] $\lambda_{1,p}^{min}(\Omega)$ is computed explicitly;

\item[(III)] The set of extremizers of $\lambda_{1,p}^{\min}(\Omega)$ in $\s({\cal H})$, if non-empty, is contained in ${\cal H}^D$ and is completely characterizable. In particular, its cardinality (i.e. multiplicity of anisotropic extremizers of \eqref{LII}) corresponds to the quantity of directions in which the width of $\Omega$ is maximum;

\item[(IV)] $\lambda_{1,p}^{\max}(\Omega)$ is equal to the fundamental frequency $\lambda_{1,p}(\Omega)$ associated to the usual $p$-Laplace operator;

\item[(V)\ ] The set of extremizers of $\lambda_{1,p}^{\max}(\Omega)$ in $\s({\cal H})$ is single. In other words, the inequality \eqref{UII} is rigid.
\end{itemize}
For (I), (II) and (III) we refer to Theorem \ref{T4} and for (IV) and (V) to Theorem \ref{T3}.

\subsection{Sharp uniform estimates for fundamental frequencies}

The anisotropic optimization problems proposed in the previous subsection are closely related to the spectral program \ref{itemA} and \ref{itemB} described in the overview subsection. Indeed, given a number $p > 1$ and a membrane $\Omega$, we set

\[
{\cal E}(\Omega) := \{-\Delta_p^H:\ H \in {\cal H}\}\ \ {\rm and}\ \ {\cal M}(-\Delta_p^H) := \Vert H \Vert^p.
\]
From Property \ref{PropertyP1}, the sharp inequalities \eqref{LII} and \eqref{UII} can be translated to the sharp uniform estimates

\begin{equation} \label{SUE}
\lambda_{1,p}^{\min}(\Omega)\, {\cal M}(-\Delta_p^H) \leq \lambda_{1,p}^H(\Omega) \leq \lambda_{1,p}^{\max}(\Omega)\, {\cal M}(-\Delta_p^H) \tag{U-estimate}
\end{equation}
for every $H$-anisotropic $p$-Laplace operator $-\Delta_p^H \in {\cal E}(\Omega)$.

Among many operators in ${\cal E}(\Omega)$, we select some more canonical ones:

\begin{ex} \label{ex4}
Linear elliptic operators $(p=2)$:

$\blacktriangleright$ (rather degenerate operators) ${\cal L}u = - \frac{\partial^2 u}{\partial x^2}$ and ${\cal L}u = - \frac{\partial^2 u}{\partial y^2}$;

$\blacktriangleright$ (non-degenerate operators) ${\cal L}u = - \Delta u$ (Laplace operator) and ${\cal L}u = - {\rm div}(A \nabla u)$ (operators with constant coefficients in divergence form).
\end{ex}
\n The $2D$ anisotropies associated to each of these linear examples are respectively:

$\triangleright$ $H(x,y) = \vert x \vert$ and $H(x,y) = \vert y \vert$;

$\triangleright$ $H(x,y) = \vert (x,y) \vert$ and $H(x,y) = \vert (x,y)^T A\, (x,y) \vert^{\frac 12}$ for invertible symmetric matrices $A$.

\begin{ex} \label{ex5}
Quasilinear elliptic operators $(p \neq 2)$:

$\blacktriangleright$ (rather degenerate operators) ${\cal L}u = - \frac{\partial}{\partial x}\left( \vert \frac{\partial u}{\partial x} \vert^{p-2} \frac{\partial u}{\partial x} \right)$ ($p$-Laplace operator on $x$) and ${\cal L}u = - \frac{\partial}{\partial y}\left( \vert \frac{\partial u}{\partial y} \vert^{p-2} \frac{\partial u}{\partial y} \right)$ ($p$-Laplace operator on $y$);

$\blacktriangleright$ (degenerate operators) ${\cal L}u = - \Delta_p u$ ($p$-Laplace operator), ${\cal L}u = - \bar{\Delta}_p u$ (pseudo $p$-Laplace operator) and ${\cal L}u = - \Delta_p u - \bar{\Delta}_p u$ (mixed anisotropic operator).
\end{ex}
\n The $2D$ anisotropies associated to each of these nonlinear examples are respectively:

$\triangleright$ $H(x,y) = \vert x \vert$ and $H(x,y) = \vert y \vert$;

$\triangleright$ $H(x,y) = \vert (x,y) \vert$, $H(x,y) = \left( \vert x \vert^p + \vert y \vert^p \right)^{1/p}$ and $H(x,y) = \left( \vert (x,y) \vert^p + \vert x \vert^p + \vert y \vert^p \right)^{1/p}$.\\

\n Lastly, we point out that the list of contributions (I)-(V) mentioned in the previous subsection furnishes a complete solution for the spectral program \ref{itemA} and \ref{itemB} within the homogeneous anisotropic context. In effect, according to the discussion made in Subsection \ref{ssec: Anisotropic least levels versus fundamental frequencies}, both estimates in \eqref{SUE} are sharp for fundamental frequencies associated to uniformly elliptic operators $\Delta_p^H$, whose anisotropy $H \in {\cal H}^P$ we know that necessarily is of $C^1$ class (see Theorem \ref{T6}), with explicit optimal constants $\lambda_{1,p}^{\min}(\Omega)$ and $\lambda_{1,p}^{\max}(\Omega)$. Furthermore, for this class of operators on a slightly smooth membrane $\Omega$, the lower estimate in \eqref{SUE} is always strict, since equality only can occurs for degenerate anisotropic operators, while the upper equality is valid only for multiples of the $p$-Laplace operator.

\subsection{Shapes optimizations}\label{ssec: Shapes optimizations}

In view of the relevance of the anisotropic constants $\lambda_{1,p}^{\min}(\Omega)$ and $\lambda_{1,p}^{\max}(\Omega)$, it is strategic and natural to investigate estimates of them in terms of the membrane $\Omega$. In this subsection we choose to explore two categories of shape optimization associated to these two optimal constants.

The first category is isodiametric in nature and consists of the optimization problems:

\begin{align*}
& \inf\{\lambda_{1,p}^{\min}(\Omega):\, \forall\, \text{membrane\ } \Omega \subset \R^2\ {\rm with}\ {\rm diam}(\Omega) = 1\}, \tag{ID1}\label{ID1} \\
& \sup\{\lambda_{1,p}^{\min}(\Omega):\, \forall\, \text{membrane\ } \Omega \subset \R^2\ {\rm with}\ {\rm diam}(\Omega) = 1\}, \tag{ID2}\label{ID2} \\
& \inf\{\lambda_{1,p}^{\max}(\Omega):\, \forall\, \text{membrane\ } \Omega \subset \R^2\ {\rm with}\ {\rm diam}(\Omega) = 1\}, \tag{ID3}\label{ID3} \\
& \sup\{\lambda_{1,p}^{\max}(\Omega):\, \forall\, \text{membrane\ } \Omega \subset \R^2\ {\rm with}\ {\rm diam}(\Omega) = 1\}. \tag{ID4}\label{ID4}
\end{align*}
As a consequence of our developments, it is possible to guarantee that:

\begin{itemize}
\item[(D1)] The value of \eqref{ID1} is positive and the corresponding optimal shapes consist of membranes that contain at least one open segment whose length is equal to $1$;
\item[(D2)] The value of \eqref{ID2} is infinite;
\item[(D3)] The value of \eqref{ID3} is positive and the corresponding optimal shapes are disks. The underlying isodiametric inequality is stronger than the inequality \eqref{FK} for $\lambda_{1,p}(\Omega)$;
\item[(D4)] The value of \eqref{ID4} is infinite.
\end{itemize}

The second one concerns the isoperimetric optimization problems:

\begin{align*}
& \inf\{\lambda_{1,p}^{\min}(\Omega):\, \forall\, \text{membrane\ } \Omega \subset \R^2\ {\rm with}\ \vert \Omega \vert = 1\}, \tag{IP1}\label{IP1} \\
& \sup\{\lambda_{1,p}^{\min}(\Omega):\, \forall\, \text{membrane\ } \Omega \subset \R^2\ {\rm with}\ \vert \Omega \vert = 1\}, \tag{IP2}\label{IP2} \\
& \inf\{\lambda_{1,p}^{\max}(\Omega):\, \forall\, \text{membrane\ } \Omega \subset \R^2\ {\rm with}\ \vert \Omega \vert = 1\}, \tag{IP3}\label{IP3} \\
& \sup\{\lambda_{1,p}^{\max}(\Omega):\, \forall\, \text{membrane\ } \Omega \subset \R^2\ {\rm with}\ \vert \Omega \vert = 1\}. \tag{IP4}\label{IP4}
\end{align*}
Similarly as above, our tools allow us to provide complete answers to each of them. Indeed, it is possible to deduce that:

\begin{itemize}
\item[(P1)] The value of \eqref{IP1} is zero, so the infimum is never attained for any membrane;
\item[(P2)] The value of \eqref{IP2} is positive on convex bodies and the corresponding optimal shapes are disks;
\item[(P3)] The value of \eqref{IP3} is positive and the problem is equivalent to the inequality \eqref{FK} for $\lambda_{1,p}(\Omega)$;
\item[(P4)] The value of \eqref{IP4} is infinite.
\end{itemize}
Among the above eight problems connected to the study of optimal anisotropic constants, we point out that only Problems \eqref{ID1}, \eqref{ID3}, \eqref{IP2} and \eqref{IP3} give rise to non-vacuum sharp inequalities, the first two being isodiametric and the other two isoperimetric.

The novelties here are truly those resulting from \eqref{ID1} and \eqref{IP2} (see Theorem \ref{T5}). In a precise manner, the first new sharp inequality states, for any membrane $\Omega$, that

\begin{equation} \label{ID-min}
 \lambda_{1,p}^{\min}(\Omega) \geq \lambda_{1,p}(0,1)\, {\rm diam}(\Omega)^{-p} \tag{ID-min}
\end{equation}
and the second ensures, for any convex membrane $\Omega$, that

\begin{equation} \label{IP-min}
\lambda_{1,p}^{\min}(\Omega) \leq \lambda_{1,p}(0,2/\sqrt{\pi})\, \vert \Omega \vert^{-p/2}. \tag{IP-min}
\end{equation}
Surprisingly, both inequalities combined with the scaling property $\lambda_{1,p}(0,L) = L^{-p} \lambda_{1,p}(0,1)$ yield the famous isodiametric inequality from the convex geometry:

\begin{equation} \label{ID}
\vert \Omega \vert \leq \pi \left( \frac{{\rm diam}(\Omega)}{2} \right)^2
\end{equation}
whose equality holds if and only if $\Omega$ is a disk, see for example \cite{Gr} for different proofs of it. In particular, \eqref{ID-min} and \eqref{IP-min} can independently be viewed as stronger variants of \eqref{ID}, once the validity of the latter passes easily from convex to non-convex membranes by means of convex hull.

Consider now a disk $D_0 \subset \R^2$ centered at the origin with unit diameter. The sharp inequality equivalent to Problem \eqref{ID3} asserts, for any membrane $\Omega$, that

\[
\lambda_{1,p}^{\max}(\Omega) \geq \lambda_{1,p}(D_0)\, {\rm diam}(\Omega)^{-p}.
\]
This inequality has been proved for $p=2$ by Bogosel, Henrot and Lucardesi in \cite{BHL} since $\lambda_{1,p}^{\max}(\Omega) = \lambda_{1,p}(\Omega)$. Their argument works step by step for any $p > 1$ once the key ingredients are \eqref{ID} and the inequality \eqref{FK} for $\lambda_{1,p}(\Omega)$, see \cite{AFT}.

Finally, as mentioned before, Problem \eqref{IP3} is just equivalent to the inequality \eqref{FK} for $\lambda_{1,p}(\Omega)$ for any $p > 1$.

We close this subsection with a brief justification of the remainder statements related to the shape optimization problems introduced above.\\

\n {\bf On Problem \eqref{ID2}:}

\begin{itemize}
\item[$\bullet$] (convex membranes) Let $\Omega$ be any convex membrane with ${\rm diam}(\Omega) = 1$. Theorem \ref{T4} ensures that $\lambda_{1,p}^{\min}(\Omega) = \lambda_{1,p}(0,1)$, that is, the problem trivializes in this case.

\item[$\bullet$] (non-convex membranes) Consider as example the annulus $\Omega_k = D_0 \setminus \bar{D}_k$, where $D_k$ is the disk concentric to $D_0$ with radius $r_k < 1/2$ where $r_k$ converges to $1/2$. Clearly, ${\rm diam}(\Omega_k) = 1$. By Theorem \ref{T2}, we have $\lambda_{1,p}^H(\Omega_k) = \lambda_{1,p}(0, \sqrt{1 - 4 r^2_k})$ for any $H \in {\cal H}^D$. Hence, by Theorem \ref{T4}, $\lambda_{1,p}^{\min}(\Omega_k) = \lambda_{1,p}(0, \sqrt{1 - 4r^2_k}) \rightarrow \infty$ as $k \rightarrow \infty$.
\end{itemize}

\n {\bf On Problem \eqref{ID4}:}

\begin{itemize}
\item[$\bullet$] (convex membranes) For $H(x,y) = \vert y \vert$ and $\Omega_k = (0,1) \times (0,1/k)$, we have ${\rm diam}(\Omega_k) = 1$ and, by Theorem \ref{T2}, $\lambda_{1,p}^{\max}(\Omega_k) \geq \lambda_{1,p}^H(\Omega_k) = \lambda_{1,p}(0,1/k) \rightarrow \infty$ as $k \rightarrow \infty$.

\item[$\bullet$] (non-convex membranes) Same example taken in \eqref{ID2}.
\end{itemize}

\n {\bf On Problem \eqref{IP1}:}

\begin{itemize}
\item[$\bullet$] (convex membranes) For $H(x,y) = \vert y \vert$ and $\Omega_k = (0, 1/k) \times (0,k)$, we have $\vert \Omega_k \vert = 1$ and, by Theorem \ref{T2}, $\lambda_{1,p}^{\min}(\Omega_k) \leq \lambda_{1,p}^H(\Omega_k) = \lambda_{1,p}(0,k) \rightarrow 0$ as $k \rightarrow \infty$.

\item[$\bullet$] (non-convex membranes) Let $H(x,y) = \vert y \vert$. Take a sector of annulus $\Omega_k$ with smaller radius equals to $1$, greater radius equals to $k$ and sectorial angle $\theta_k = 2 (k^2 - 1)^{-1}$ such that the axis $y$ is the bisector of $\theta_k$. From this choice of $\theta_k$, we have $\vert \Omega_k \vert = 1$ and, for $k$ large enough, one easily checks that $\lambda_{1,p}^{\min}(\Omega_k) \leq \lambda_{1,p}^H(\Omega_k) = \lambda_{1,p}(0,k - 1) \rightarrow 0$ as $k \rightarrow \infty$.
\end{itemize}

\n {\bf On Problem \eqref{IP2}:}

\begin{itemize}
\item[$\bullet$] (non-convex membranes) Let $\Omega_k$ be the membrane constructed as follows. We begin with two half annulus with smaller radius equals to $\frac{1}{6k\sqrt{\pi}}$ and greater radius equals to $\frac{1}{3k\sqrt{\pi}}$. We then glue them together in order to obtain an $S$-shaped region. We now proceed successively with these glues in order to form the membrane $\Omega_k$ in total with $12k^2$ of these $S$ type regions. It is clear that

    \[
    \vert \Omega_k \vert = \left(\frac{1}{9k^2 \pi} - \frac{1}{36k^2 \pi}\right) \pi \times 12k^2 = 1.
    \]
    On the other hand, by Theorem \ref{T4}, $\lambda_{1,p}^{\min}(\Omega_k) \geq \lambda_{1,p}(0,L_k)$, where $L_k = \frac{2}{3k\sqrt{\pi}}$ (twice the largest radius). Since $L_k \rightarrow 0$, we deduce that $\lambda_{1,p}^{\min}(\Omega_k) \rightarrow \infty$ as $k \rightarrow \infty$.
\end{itemize}

\n {\bf On Problem \eqref{IP4}:}

\begin{itemize}
\item[$\bullet$] (convex membranes) For $H(x,y) = \vert y \vert$ and $\Omega_k = (0,k) \times (0,1/k)$, we have $\vert \Omega_k \vert = 1$ and $\lambda_{1,p}^{\max}(\Omega_k) \geq \lambda_{1,p}^H(\Omega_k) = \lambda_{1,p}(0,1/k) \rightarrow \infty$ as $k \rightarrow \infty$.

\item[$\bullet$] (non-convex membranes) It follows from \eqref{IP2}.
\end{itemize}

\subsection{Main statements}\label{ssec: Main statements}

Our first theorem characterizes the shapes $\Omega$ for which $\lambda_{1,p}^H(\Omega)$ represents a fundamental frequency in the degenerate case. Here, thanks to the $C^1$ regularity of $H^p$ for degenerate anisotropies $H$ provided in Proposition \ref{P1}, the level $\lambda_{1,p}^H(\Omega)$ is a fundamental frequency if and only if it is attained in $W_0^{1,p}(\Omega)$.

For the statement, we make use of the notations $\Omega_{A} = A(\Omega)$ and $H_{A} = H \circ A$ for a rotation matrix $A$ and also $\lambda_{1,p}(0,L)$ stands for the fundamental frequency of the one-dimensional $p$-Laplace operator on the interval $(0,L)$.

\begin{teor}[shapes vs fundamental frequencies] \label{T1}
Let $\Omega$ be a $C^{0,1}$ membrane, let $p > 1$ and let $H \in {\cal H}^D\setminus\{0\}$. Let also $A$ be any rotation matrix such that $H_{A}(x,y) = c|y|$ for some constant $c > 0$. The assertions are equivalent:

\begin{itemize}
\item[(a)] $\lambda^H_{1,p}(\Omega)$ is a fundamental frequency (i.e. it admits a membrane eigenfunction in $W^{1,p}_0(\Omega)$);

\item[(b)] There are bounded open intervals $I^\prime \subset I \subset \mathbb{R}$, a $C^{0,1}$ sub-membrane $\Omega^\prime \subset \Omega$ and a number $L > 0$ such that:

\begin{itemize}

\item[(i)] $\Omega_A \subset I \times \mathbb{R}$,

\item[(ii)] Each connected component of $\{ y \in \mathbb{R}:\, (x,y) \in \Omega_A) \}$ has length at most $L$ for every $x \in I$,

\item[(iii)] The set $\{ y \in \mathbb{R}:\, (x,y) \in \Omega^\prime_A \}$ is an interval of length $L$ for every $x \in I^\prime$.
\end{itemize}
\end{itemize}
In that case, $L$ is the number so that $\lambda^H_{1,p}(\Omega) = c^p \lambda_{1,p}(0,L)$.
\end{teor}
\n Note that there are precisely two rotation matrices $A$ that transform $H(x,y)$ into $c \vert y \vert$ and one is symmetric of the other, so the assertions $(i)-(iii)$ do not depend on the choice of $A$. Furthermore, proceeding with a rotation of $\pi/2$, one observes that each of them can be reformulated for the variable $x$, so that $(i)-(iii)$ also occur in the direction $x$.

The above characterization allows to determinate the exact number $m$ of anisotropies $H \in {\cal H}^D$ normalized by $\Vert H \Vert = 1$ such that $\lambda^H_{1,p}(\Omega)$ is a fundamental frequency depending on the shape $\Omega$.

Let us illustrate some shapes $\Omega$ and corresponding numbers $m$ in the following figures:

\begin{figure}[h]
    \centering
    \begin{minipage}{0.2\textwidth}
        \centering
        \includegraphics[width=0.5\textwidth]{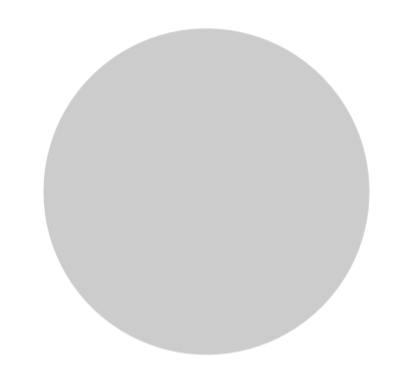} 
        \caption*{{\scriptsize Circular shape: $m=0$}}
    \end{minipage}\hspace{0.3cm}
    \begin{minipage}{0.25\textwidth}
    \vspace{0.3cm}
        \centering
        \includegraphics[width=0.4\textwidth]{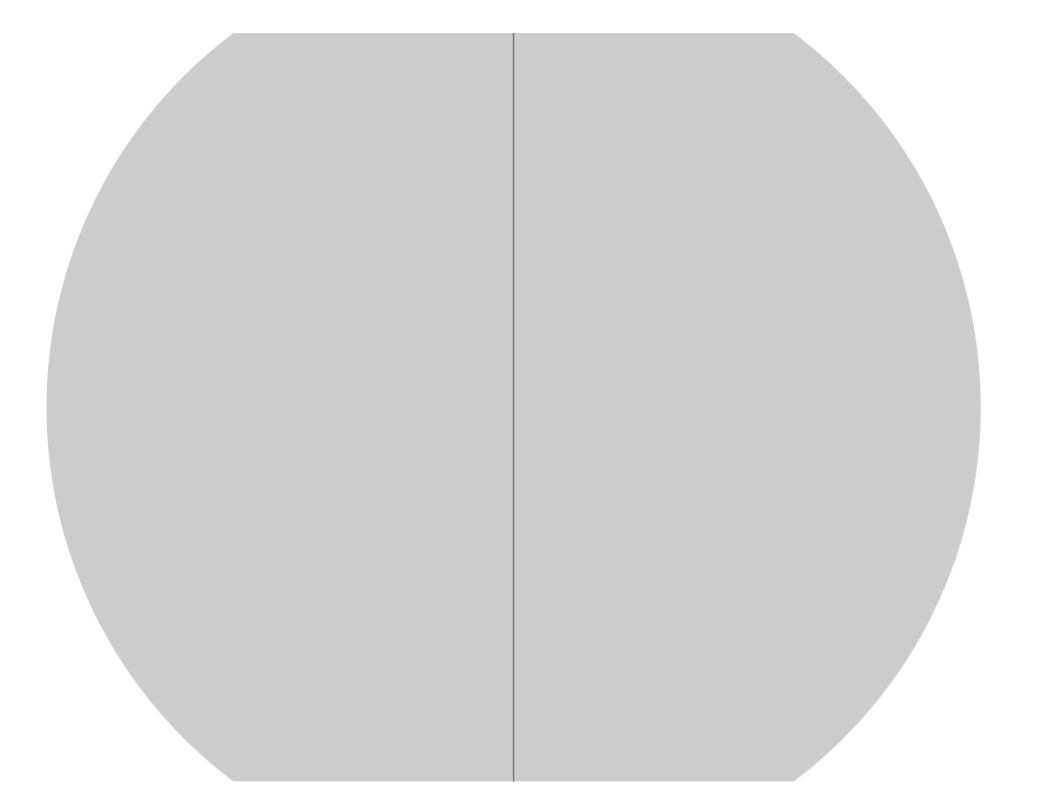} 
        \caption*{{\scriptsize Cropped circular shape: $m=1$}}
    \end{minipage}\hspace{0.3cm}	
    \begin{minipage}{0.19\textwidth}
        \centering
        \includegraphics[width=0.45\textwidth]{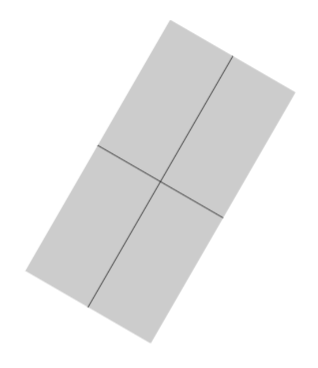} 
        \caption*{{\scriptsize Rectangular shape: $m=2$}}
    \end{minipage}
    \begin{minipage}{0.21\textwidth}
        \centering
        \includegraphics[width=0.5\textwidth]{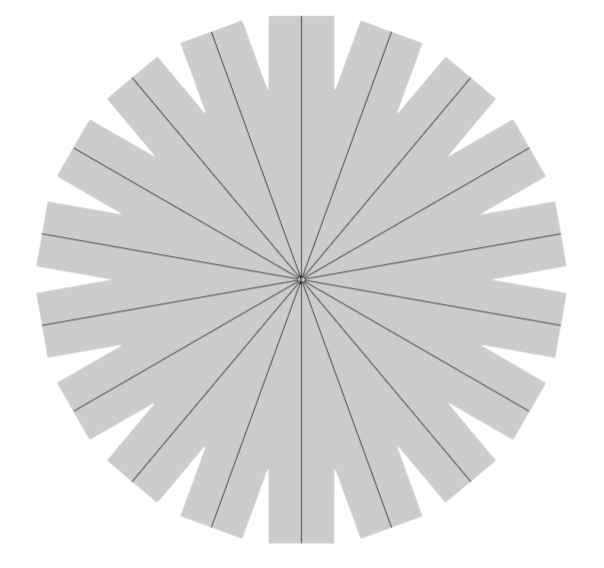} 
        \caption*{{\scriptsize Asterisk shape: $m=9$}}
    \end{minipage}
\end{figure}
\n It deserves to be noted that for $m \geq 3$ the asterisk shape is constructed from the intersection of $m$ appropriately positioned rectangles.


Given a membrane $\Omega$, consider the directional width function $L: [0,\pi] \rightarrow \R$ defined by $L(\theta) = L_\theta$, where

\[
L_\theta:= \sup_{v \in \R^{2}} \sup\left\{|J| \colon J \subset \left\{t(-\cos{\theta},\sin{\theta}) + v \colon t \in \R\right\}\cap\Omega, \ J \ \text{is connected}\right\}.
\]
It is clear that $L$ is well defined and extends $\pi$-periodically to $\R$ since $\Omega$ is a bounded domain.

Regardless of whether $\lambda^H_{1,p}(\Omega)$ is a fundamental frequency or not when $H \in {\cal H}^D\setminus\{0\}$, we determine its value in any situation.

\begin{teor}[anisotropic least level] \label{T2}
Let $\Omega$ be any membrane and let $p > 1$. For any $H \in {\cal H}^D\setminus\{0\}$, we have

\[
\lambda^H_{1,p}(\Omega) = c^p \lambda_{1,p}(0,L_{\theta}),
\]
where $c > 0$ and $\theta \in [0,\pi]$ are such that $H(x,y) = c\, \vert\cos{\theta}\, x + \sin{\theta}\, y \,\vert$.
\end{teor}

This theorem is a powerful tool that allows us to solve a long-standing conjecture. It has long been known that fundamental frequencies $\lambda_1(\Omega)$ of membranes $\Omega$ satisfy

\[
\sup \{\lambda_1(\Omega):\, \forall\, \text{membrane\ } \Omega \subset \R^2\ {\rm with}\ \vert \Omega \vert = 1\} = \infty.
\]
Its proof is direct and uses merely the explicit knowing of $\lambda_1(\Omega)$, given precisely by

\[
\pi^2 \left(\frac{1}{a^2} +  \frac{1}{b^2}\right),
\]
for rectangles under the form $\Omega = (0,a) \times (0,b)$, which was computed by Rayleigh in the century XIX.

This optimization problem gave rise to the following well-known conjecture for the $p$-Laplace operator among experts in the field:\\

\n {\bf Maximization Conjecture:} For any $p \neq 2$, it holds that

\[
\sup \{\lambda_{1,p}(\Omega):\, \forall\, \text{membrane\ } \Omega \subset \R^2\ {\rm with}\ \vert \Omega \vert = 1\} = \infty.
\]

\vs{0.5cm}

\n Unfortunately, the above argument fails for values $p \neq 2$, since fundamental frequencies $\lambda_{1,p}(\Omega)$ are generally unknown for most membranes.

It is clear that the veracity of this conjecture implies that it also occurs for $\lambda^H_{1,p}(\Omega)$ where $H$ is any anisotropy in ${\cal H}^P$. However, our strategy of proof is based on the knowing of fundamental frequencies for degenerate anisotropies expressed in Theorem \ref{T2}. In particular, our approach allows to prove the conjecture in a more general setting, including even degenerate elliptic operators.

Precisely, we have:

\begin{teor}[isoperimetric maximization] \label{T2.1}
Let $p > 1$. For any $H \in {\cal H}\setminus\{0\}$, we have

\[
\sup \{\lambda^H_{1,p}(\Omega):\, \forall\, {\rm membrane\ } \Omega \subset \R^2\ {\rm with}\ \vert \Omega \vert = 1\} = \infty.
\]
\end{teor}

The next result shows that the sharp upper anisotropic inequality \eqref{UII} is rigid on smooth membranes and so the isoanisotropic maximization problem is completely solved under these conditions.

\begin{teor}[upper anisotropic rigidity] \label{T3}
Let $\Omega$ be any membrane and let $p > 1$. The optimal anisotropic constant $\lambda_{1,p}^{\max}(\Omega)$ is given by $\lambda_{1,p}(\Omega)$ and the Euclidean norm is a corresponding anisotropic extremizer. Moreover, it is unique provided that $\partial \Omega$ is $C^{1,\alpha}$.
\end{teor}

We now focus on the complete statement of the isoanisotropic minimization problem. Before, however, we shall introduce a new definition linking the membrane to a condition of length optimality as follows:

\begin{defi}
A membrane $\Omega$ is said to have an optimal anisotropic design if the width function $L: [0,\pi] \rightarrow \R$ has a global maximum point.
\end{defi}
For a broad set of membranes $\Omega$, which includes all convex bodies and annulus as well as many non-convex membranes, the definition is clearly satisfied. More specifically, if for some couple of points $A, B \in \partial \Omega$ such that

\[
\bar{L} := \sup_{\theta \in [0, \pi]} L_\theta = \vert A - B \vert,
\]
the following one-sided property occurs:\\

\n \textit{There exists a number $\varepsilon > 0$ such that, at least in one of the half-planes determinate by the support line $r$ of $A$ and $B$, the intersection of $\Omega$ with all lines $s$ parallel to $r$ such that ${\rm dist}(s,r) < \varepsilon$ are connected sets},\\

\n then $\Omega$ is a membrane with optimal anisotropic design.

On the other hand, the figure below shows a punctual counterexample of a membrane where the optimal design condition fails.

\begin{figure}[h]
\centering
\includegraphics[width=0.3\linewidth]{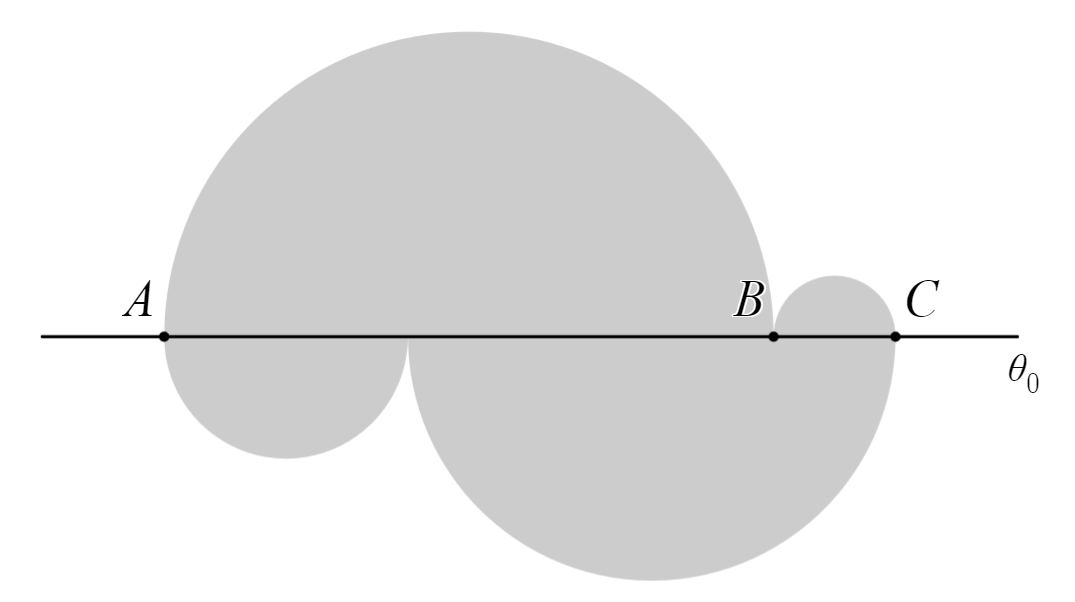}
\caption*{}
\end{figure}

\vspace{-1cm}

\n For this shape we have $\bar{L} = \vert A - C \vert$, while $L_{\pi-\theta_0} = \vert A - B \vert$, so the function $L$ doesn't admit any global maximum point.

The notion of optimal design is driven by the following statement:

\begin{teor}[lower anisotropic classification] \label{T4}
Let $\Omega$ be any membrane and let $p > 1$. The optimal anisotropic constant $\lambda_{1,p}^{\min}(\Omega)$ is always positive and given by

\[
\lambda_{1,p}^{\min}(\Omega) = \inf_{\theta \in [0,\pi]} \lambda_{1,p}(0, L_\theta).
\]
Moreover, the infimum is attained if and only if $\Omega$ has optimal anisotropic design. In this case, if $\theta_0$ is a global maximum point of the function $L$, then

\[
H(x,y) = |x\cos{\theta_0} + y\sin{\theta_0}|
\]
is an anisotropic extremizer corresponding to $\lambda_{1,p}^{\min}(\Omega)$. Furthermore, all extremizers have this form provided that $\partial \Omega$ is $C^{0,1}$.
\end{teor}

\n According to previous discussion, the optimal lower constant $\lambda_{1,p}^{\min}(\Omega)$ is never attained on $\s({\cal H})$ for the above $S$-shaped membrane. On the other hand, the set of anisotropic extremizers on $\s({\cal H})$ is a nonempty subset of ${\cal H}^D$ for any $C^{0,1}$ membrane with optimal anisotropic design and in addition their cardinality is equal to the number of points of global maximum of the function $L$. In particular, the classification of optimal $2D$ anisotropies given in Theorem \ref{T4} provides a close relationship between shapes $\Omega$ and multiplicity of anisotropic extremizers for $\lambda_{1,p}^{\min}(\Omega)$.

For an integer $m \geq 1$, the next figure shows how to construct a membrane $\Omega$ for which $\lambda_{1,p}^{\min}(\Omega)$ admits precisely $m$ extremizers in $\s({\cal H})$:

\begin{figure}[h]

\centering
\includegraphics[width=0.13\linewidth]{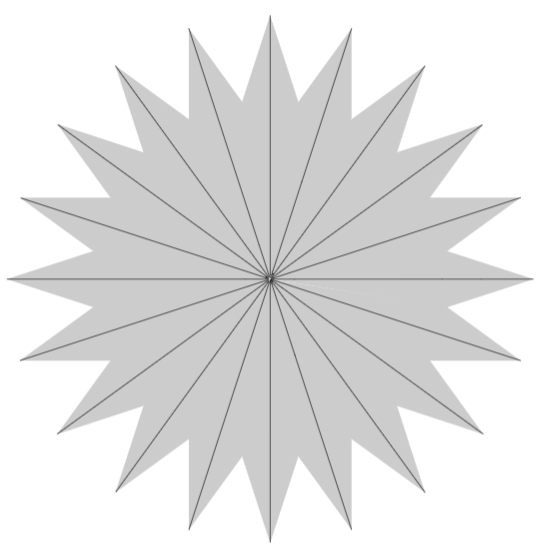}
\caption*{{\scriptsize Star shape: $m=10$}}
\end{figure}

\n It is also interesting to point out in this example that $\lambda_{1,p}^H(\Omega)$ is never a fundamental frequency for any anisotropy $H \in {\cal H}^D \setminus \{0\}$.







Finally, we close the section with the two sharp shapes inequalities which provide upper and lower controls of the optimal lower constant $\lambda_{1,p}^{\min}(\Omega)$ in terms of $\Omega$. Indeed, consider the following inequalities introduced in Subsection \ref{ssec: Shapes optimizations} for any $p > 1$:

\begin{equation}
\lambda_{1,p}^{\min}(\Omega) \geq \lambda_{1,p}(0,1)\, {\rm diam}(\Omega)^{-p} \tag{ID-min}
\end{equation}

\begin{equation}
\lambda_{1,p}^{\min}(\Omega) \leq \lambda_{1,p}(0,2/\sqrt{\pi})\, \vert \Omega \vert^{-p/2} \tag{IP-min}
\end{equation}

For these inequalities, we ensure that:

\begin{teor}[shape optimization] \label{T5}
The sharp inequality \eqref{ID-min} holds for any membrane $\Omega$, while the sharp inequality \eqref{IP-min} holds only for convex membranes $\Omega$. Moreover:

\begin{itemize}
\item[(i)] Equality holds in \eqref{ID-min} if and only if ${\rm diam}(\Omega) = \sup_{\theta \in [0, \pi]} L_\theta$;
\item[(ii)] Equality holds in \eqref{IP-min} if and only if $\Omega$ is a disk.
\end{itemize}
\end{teor}
\n The condition in (i) is clearly satisfied for any convex membrane and for many non-convex ones, such as the above $S$-shaped membrane, whereas it does not hold for any annulus and also for some conveniently constructed membranes without any holes.

\section{Shapes versus eigenfunctions: the degenerate case}

This section is dedicated to detailed proofs of Theorems \ref{T1}, \ref{T2} and \ref{T2.1}.

\subsection{A result on dimension reduction in PDEs}

We present a key result in the proof of Theorem \ref{T1} about trace of weak solutions of a degenerate Dirichlet problem. For a bounded domain $\Omega \subset \R^2$, consider the bounded open interval $I \subset \R$ such that $I = \{x \in \R :\, \Omega_x \neq \emptyset\}$, where $\Omega_{x} = \left\{y \in \R :\, (x,y) \in \Omega\right\}$.

\begin{propo}\label{propsol} Assume $\Omega \subset \R^2$ is any bounded domain and let $p > 1$. Let $H(x,y) = |y|$ and let $f: \R \rightarrow \R$ be a continuous function satisfying the additional growth condition if $1 < p \leq 2$:

\[
\vert f(t) \vert \leq C (\vert t \vert^q + 1),\ \forall t \in \R,
\]
where $0 \leq q \leq \frac{2p}{2 - p}$ if $p < 2$ and $q \geq 0$ if $p = 2$.

Let $u \in W^{1,p}_0(\Omega)$ be a weak solution of the equation

\begin{equation} \label{Dpde}
- \Delta_{p}^H u = f(u) \ \ \text{in} \ \ \Omega.
\end{equation}
Then, for almost every $x \in I$, the restriction

\[
u_{x}(y) = u(x,y),\ \forall y \in \Omega_{x}
\]
is a weak solution in $W^{1,p}_0(\Omega_x)$ of the one-dimensional equation

\[
-\left( |v^\prime|^{p-2} v^\prime \right)^\prime = f(v) \ \ \text{in} \ \ \Omega_{x}.
\]
\end{propo}

\begin{proof} Let $u \in W^{1,p}_0(\Omega)$ be a weak solution of \eqref{Dpde}. The initial claim that $u_x \in W^{1,p}_0(\Omega_x)$ for almost every $x \in I$ follows readily from Theorem 10.35 of \cite{Le}. By Morrey's Theorem, it is clear that $f(u_x) \in L^1(\Omega_x)$ for almost every $x \in I$. Besides, the growth assumption assumed on $f$ guarantees that $f(u) \in L^1(\Omega)$ for any $p > 1$.

We next organize the proof of the second claim into two steps.

Assume first that the boundary of $\Omega$ is of $C^1$ class. For each $x \in I$, the set $\Omega_{x}$ is a finite union of open intervals, so it is enough to prove for the case that $\Omega_{x}$ is an interval to be denoted by $(c_x,d_x)$. Using that $\partial \Omega$ is $C^1$, then the functions


\begin{equation} \label{functionstrip}
x \mapsto c_x\ \ {\rm and}\ \ x \mapsto d_x
\end{equation}
are piecewise $C^1$ in $I$. For each $x \in I$, consider the map

\begin{align*}
\Gamma_x \colon & C^{\infty}_{0}(0,1) \longrightarrow C^{\infty}_{0}\left(\Omega_x\right) \\
& \varphi \longmapsto \varphi\left(\frac{\mathord{\cdot}-c_x}{d_x - c_x}\right).
\end{align*}
Clearly, $\Gamma_x$ is a bijection. Take now any $\varphi \in C^{\infty}_{0}(0,1)$ and choose $x_{0} \in I$ and $\delta_{x_{0}} > 0$ so that the functions in \eqref{functionstrip} are $C^1$ in $(x_{0}-\delta_{x_{0}}, x_{0}+\delta_{x_{0}})$ and in addition

\[
d_x - c_x \geq \frac{d_{x_{0}}-c_{x_{0}}}{2} > 0,\ \ \forall x \in (x_{0}-\delta_{x_{0}}, x_{0}+\delta_{x_{0}}).
\]
Let any $\eta \in C^{\infty}_{0}(x_{0}-\delta_{x_{0}},x_{0}+\delta_{x_{0}})$ and define the function $\phi(x,y) = \eta(x) \Gamma_{x}(\varphi(y))$. Notice that $\phi \in C^{1}_{0}(\Omega)$ and so taking $\phi$ as a test function in \eqref{Dpde}, we have

\[
\int_{x_{0}-\delta_{x_{0}}}^{x_{0}+\delta_{x_{0}}} \int_{\Omega_{x}}\vert \frac{\partial u}{\partial y}(x,y) \vert^{p-2}\frac{\partial u}{\partial y}(x,y) \frac{\eta(x) \Gamma_{x}(\varphi^\prime(y))}{d_{x} - c_{x}}\, dydx  = \int_{x_{0}-\delta_{x_{0}}}^{x_{0}+\delta_{x_{0}}} \int_{\Omega_{x}}f(u(x,y)) \eta(x) \Gamma_{x}(\varphi(y))\, dydx.
\]
On the other hand, this equality can be rewritten as

\[
\int_{x_{0}-\delta_{x_{0}}}^{x_{0}+\delta_{x_{0}}}\left[\int_{\Omega_{x}}\vert \frac{\partial u}{\partial y}(x,y) \vert^{p-2}\frac{\partial u}{\partial y}(x,y)\frac{\Gamma_{x}(\varphi^\prime(y))}{d_{x} - c_{x}}\, dy - \int_{\Omega_{x}}f(u(x,y))\Gamma_{x}(\varphi(y)) \, dy\right] \eta(x)\, dx = 0
\]
for all $\eta \in C^{\infty}_{0}(x_{0}-\delta_{x_{0}},x_{0}+\delta_{x_{0}})$. Since the function

\[
x \mapsto \int_{\Omega_{x}}\vert \frac{\partial u}{\partial y}(x,y) \vert^{p-2}\frac{\partial u}{\partial y}(x,y)\frac{ \Gamma_{x}(\varphi^\prime(y))}{d_{x}-c_{x}} \, dy - \int_{\Omega_{x}}f(u(x,y)) \Gamma_{x}(\varphi(y))\, dy
\]
belongs to $L^{1}_{loc}(x_{0} - \delta_{x_{0}}, x_{0}+\delta_{x_{0}})$, we must have

\[
\int_{\Omega_{x}} \vert \frac{\partial u}{\partial y}(x,y) \vert^{p-2}\frac{\partial u}{\partial y}(x,y)\frac{ \Gamma_{x}(\varphi^\prime(y))}{d_{x}-c_{x}}\, dy - \int_{\Omega_{x}}f(u(x,y)) \Gamma_{x}(\varphi(y))\, dy = 0
\]
for almost every $x$ in $(x_{0} - \delta_{x_{0}}, x_{0}+\delta_{x_{0}})$.

Consequently, the function $u_{x}(y) = u(x,y)$ satisfies

\[
\int_{\Omega_{x}}|u_{x}^\prime|^{p-2}u_{x} \psi^\prime\, dy = \int_{\Omega_{x}}f(u_{x}) \psi\, dy
\]
for every $\psi \in C^{\infty}_{0}(\Omega_{x})$ for almost every $x$ in $(x_{0} - \delta_{x_{0}}, x_{0}+\delta_{x_{0}})$. Since we can cover $I$, module a countable set, by an enumerate union of intervals like $(x_{k} - \delta_{x_{k}}, x_{k} + \delta_{x_{k}})$, we get

\[
\int_{\Omega_{x}}|u_{x}^\prime|^{p-2}u_{x} \psi^\prime\, dy = \int_{\Omega_{x}}f(u_{x}) \psi\, dy
\]
for every $\psi \in C^{\infty}_{0}(\Omega_{x})$ for almost every $x$ in $I$.

Consider now $\Omega$ an arbitrary bounded domain. One always can write

\[
\Omega = \bigcup_{j=1}^{\infty} \Omega_{j},
\]
where $\Omega_{j}$ are bounded domains with $C^1$ boundary and $\Omega_{j} \subseteq \Omega_{j+1}$ for all $j \in \N$. If $\psi \in C^{\infty}_{0}(\Omega_{x})$, then there is $j_{0} \in \N$ such that the support of $\psi$ is contained in $(\Omega_{j})_{x}$ for all $j \geq j_{0}$. Thus, from the first part, $u_{x}(y) = u(x,y)$ satisfies

\[
\int_{(\Omega_{j})_{x}}|u_{x}^\prime|^{p-2} u_{x}^\prime \psi^\prime\, dy = \int_{(\Omega_{j})_{x}} f(u_{x}) \psi\, dy
\]
for almost every $x$ in $I$ (module a countable union of countable sets) and all $j \geq j_{0}$. Taking $j \to \infty$ in the above equality, we derive

\[
\int_{\Omega_{x}}|u_{x}^\prime|^{p-2}u_{x}^\prime \psi^\prime\, dy = \int_{\Omega_{x}} f(u_{x})\psi\, dy
\]
almost everywhere for $x \in I$, which proves the proposition.
\end{proof}

\subsection{Characterization of fundamental frequencies}

We now concentrate on the proof of Theorem \ref{T1} which states that the items $(i)$, $(ii)$ and $(iii)$ in the part $(b)$ are necessary and sufficient for $\lambda^H_{1,p}(\Omega)$ to be a fundamental frequency.

\begin{proof}[Proof of Theorem \ref{T1}] Let $\Omega \subset \R^2$ be a membrane and let $H \in {\cal H}^D \setminus \{0\}$. By Proposition \ref{P1}, one knows that $H(x,y) = c\left| \cos \theta\, x + \sin \theta\, y \right|$ for some $c > 0$ and $\theta \in [0,\pi]$. Let $A$ be a rotation matrix such that $H_{A}(x,y) = c \vert y \vert$. By Property \ref{PropertyP4}, for convenience we assume that $H = H_{A}$ and $\Omega = \Omega_{A}$. By Property \ref{PropertyP1}, we also consider $c = 1$.

For a bounded domain $\Omega$, the existence of a bounded open interval $I$ which satisfies $(i)$ is clearly assured. Note also that $\Omega_x := \{ y \in \R:\, (x,y) \in \Omega\}$ is a nonempty open set for every $x \in I$. Since $H$ is nonzero, by Proposition \ref{P.1}, $\lambda^H_{1,p}(\Omega)$ is positive. So, there is a number $L > 0$ such that $\lambda^H_{1,p}(\Omega) = \lambda_{1,p}(0,L)$.

For such choices of $I$ and $L$, we show $(ii)$ by contradiction. Assume that there is a point $x_{0} \in I$ such that $\Omega_{x_0}$ contains a connected component $\tilde{\Omega}_{x_0}$ whose measure is greater than $L$. Then, we can take a rectangle $R_1 = (a_1, b_1) \times (c_1, d_1) \subset \Omega$ with $L_1 := d_1 - c_1 > L$. Let $\psi_p \in W^{1,p}_0(c_1, d_1)$ be an eigenfunction of the one-dimensional $p$-laplacian corresponding to $\lambda_{1,p}\left( 0,L_1 \right)$ normalized by $\Vert \psi_p \Vert_p = 1$ and let $\varphi \in W^{1,p}_0(a_1, b_1)$ be any function satisfying $\Vert \varphi \Vert_p = 1$. Define $u(x,y) := \varphi(x) \psi_p(y)$. It is clear that $u \in W_0^{1,p}(R_1) \subset W_0^{1,p}(\Omega)$ and $\Vert u \Vert_p = 1$. In addition,

\begin{eqnarray*}
\iint_{\Omega} H^p(\nabla u)\, dA &=& \int_{a_1}^{b_1} \int_{c_1}^{d_1} \vert \frac{\partial u}{\partial y}(x,y) \vert^p\, dy\, dx = \int_{a_1}^{b_1} \vert \varphi(x) \vert^p\, dx \int_{c_1}^{d_1} \vert \psi^\prime_p(y) \vert^p\, dy \\
&=& \int_{c_1}^{d_1} \vert \psi^\prime_p(y) \vert^p\, dy = \lambda_{1,p}\left(0, L_1 \right) \int_{c_1}^{d_1} \vert \psi_p(y) \vert^p\, dy = \lambda_{1,p}\left(0, L_1 \right).
\end{eqnarray*}
This gives the contradiction $\lambda^H_{1,p}(\Omega) \leq \lambda_{1,p}\left( 0, L_1 \right) < \lambda_{1,p}\left( 0, L \right)$.

We now focus on the proof of $(iii)$. Here we use the assumption that $\lambda^H_{1,p}(\Omega)$ is a fundamental frequency. Let $u_p \in W^{1,p}_0(\Omega)$ be a corresponding membrane eigenfunction. From the homogeneity of $H$, we can assume $u_p$ is nonnegative in $\Omega$.

Let $u = u_p$. We recall that
\[
-\Delta_p^H u(x,y) = -\frac{\partial}{\partial y} \left( \vert \frac{\partial u}{\partial y}(x,y)\vert^{p-2} \frac{\partial u}{\partial y}(x,y) \right) =: -\Delta_{p,y} u(x,y)\ \ {\rm in}\ \ \Omega_x,
\]
in other words, $-\Delta_p^H$ can be viewed as the one-dimensional $p$-Laplace operator $-\Delta_{p,y}$ on $\Omega_x$ for $x \in I$. So, applying Proposition \ref{propsol} to the function $f(t) = \lambda^H_{1,p}(\Omega) \vert t \vert^{p-2} t$ for $p > 1$, it follows that $u(x,\cdot) \in W^{1,p}_0(\Omega_x)$ is a nonnegative weak solution of

\[
-\Delta_{p,y} u(x,y) = \lambda^H_{1,p}(\Omega) u(x,y)^{p-1}\ \ {\rm in}\ \ \Omega_x
\]
for almost every $x \in I$. By elliptic regularity, we have $u(x,\cdot) \in C^1(\bar{\Omega}_x)$ and moreover either $u(x,\cdot) = 0$ or $u(x,\cdot) > 0$ in each connected component $\tilde{\Omega}_x$ of $\Omega_x$. In the second case, one concludes that $\lambda^H_{1,p}(\Omega)$ is precisely the fundamental frequency $\lambda_{1,p}(\tilde{\Omega}_x)$ associated to the operator $-\Delta_{p,y}$ on $\tilde{\Omega}_x$.

We next prove that there is an open subinterval $I^\prime \subset I$ such that $u(x,\cdot) > 0$ for some component $\tilde{\Omega}_x$ of $\Omega_x$ and so $\lambda_{1,p}(\tilde{\Omega}_x) = \lambda^H_{1,p}(\Omega)$ for every $x \in I^\prime$. Since $u$ is nonzero somewhere in $\Omega$, there is a subset $X \subset I$ with $\vert X \vert > 0$ (i.e. positive Lebesgue measure) such that $u(x,\cdot) > 0$ in some connected component $\tilde{\Omega}_x$ of $\Omega_x$ for every $x \in X$. Similarly to the definition of $I$, let $J$ be the bounded open interval such that $\Omega \subset \R \times J$ and $\Omega^y := \{x \in \R: \, (x,y) \in \Omega\} \neq \emptyset$ for every $y \in J$. We know that $u(\cdot,y) \in W^{1,p}_0(\Omega^y)$ and so, by Morrey's Theorem, $u(\cdot,y) \in C(\bar{\Omega}^y)$ for almost every $y \in J$.

Let $x_0 \in X$ and take $y_0 \in \tilde{\Omega}_{x_0}$ so that $u(\cdot,y_0) \in C(\bar{\Omega}^{y_0})$. Since $u(x_0,y_0) > 0$, by continuity, there is a number $\delta > 0$ such that $u(x,y_0) > 0$ for every $x \in I^\prime:= (x_0 - \delta, x_0 + \delta)$. Consequently, $u(x,\cdot) > 0$ in $\tilde{\Omega}_x$ (here $\tilde{\Omega}_x$ is the connected component of $\Omega_x$ such that $(x,y_{0})$ belongs to it), and so $\lambda_{1,p}(\tilde{\Omega}_x) = \lambda^H_{1,p}(\Omega) = \lambda_{1,p}\left(0,L \right)$ for every $x \in I^\prime$, which is equivalent to the equality $\vert \tilde{\Omega}_x \vert = L$  for every $x \in I^\prime$. Furthermore, the $C^{0,1}$ regularity of $\partial \Omega$ implies that $\Omega^\prime:= \{ (x,y) \in \R^2:\, x \in I^\prime,\, y \in \tilde{\Omega}_x \}$ is a $C^{0,1}$ sub-membrane of $\Omega$. Hence, the assertion $(iii)$ holds for the number $L$ as defined above.

Conversely, let $\Omega \subset \R^2$ be a membrane and let $H \in {\cal H}^D \setminus \{0\}$ satisfying the conditions $(i)$, $(ii)$ and $(iii)$, where $A$ is a rotation matrix such that $H_A(x,y) = c \vert y \vert$. By Property \ref{PropertyP4}, it suffices to construct a minimizer for $\lambda^{H_A}_{1,p}(\Omega_A)$ in $W_0^{1,p}(\Omega_A)$. Again we set $\Omega = \Omega_A$, $H = H_A$ and assume $c = 1$.

Let $I^\prime$ and $I$ be as in the statement $(b)$ and $\Omega_x$ as introduced above. Given any $u \in W_0^{1,p}(\Omega)$, we recall that $u(x, \cdot) \in W_0^{1,p}(\Omega_x)$ for almost every $x \in I$. Then, using $(ii)$ and the one-dimensional Poincaré inequality, we get

\begin{eqnarray*}
\iint_\Omega H^p(\nabla u)\, dA &=& \int_{I} \int_{\Omega_x} \vert \frac{\partial u}{\partial y}(x,y) \vert^p\, dy dx \\
&\geq& \int_{I} \lambda_{1,p}\left(0,L \right) \int_{\Omega_x} \vert u(x,y) \vert^p\, dy dx \\
&=& \lambda_{1,p}\left(0,L \right) \iint_\Omega \vert u(x,y) \vert^p\, dA,
\end{eqnarray*}
so that $\lambda^H_{1,p}(\Omega) \geq \lambda_{1,p}\left(0,L \right)$.

We now construct a function $u_0 \in W_0^{1,p}(\Omega)$ such that $\Vert u_0 \Vert_p = 1$ and

\[
\iint_\Omega H^p(\nabla u_0)\, dA = \lambda_{1,p}\left(0,L \right).
\]
Since the boundary of $\Omega$ is of $C^{0,1}$ class and $(iii)$ is satisfied, there are an open interval $I_0 \subset I^\prime$ and a Lipschitz function $g : I_0 \rightarrow \R$ such that $\Omega_0:= \{(x,y):\, x \in I_0,\, g(x) < y < g(x) + L\} \subset \Omega^\prime \subset \Omega$. Let $\psi_p \in W_0^{1,p}(0,L)$ be the positive eigenfunction for the one-dimensional $p$-laplacian corresponding to $\lambda_{1,p}\left(0,L \right)$ with $\Vert \psi_p \Vert_p = 1$ and take any function $\varphi \in W_0^{1,p}(I_0)$ with $\Vert \varphi \Vert_p = 1$. Define

\begin{equation*}
u_0(x,y) =
\begin{cases}
\begin{aligned}
    &\varphi(x) \psi_p(y - g(x)), \ \text{if} \ (x,y) \in \Omega_0 \\
    & 0, \ \text{if} \ (x,y) \in \Omega\setminus\Omega_{0}.
\end{aligned}
\end{cases}
\end{equation*}
Then, $u_0 \in W_0^{1,p}(\Omega_0) \subset W_0^{1,p}(\Omega)$ satisfies

\begin{eqnarray*}
\iint_\Omega H^p(\nabla u_0)\, dA &=& \iint_{\Omega_0} H^p(\nabla u_0)\, dA = \int_{I_0} \int_{g(x)}^{g(x) + L} \vert \varphi(x) \vert^p \vert \psi_p^\prime(y - g(x)) \vert^p\, dy dx \\
&=& \int_{I_0} \int_0^L \vert \varphi(x) \vert^p \vert \psi_p^\prime(y) \vert^p\, dy dx = \int_{I_0} \vert \varphi(x) \vert^p\, dx \int_0^L \vert \psi_p^\prime(y) \vert^p\, dy\\
&=& \lambda_{1,p}\left(0,L \right) \int_{I_0} \vert \varphi(x) \vert^p\, dx \int_0^L \vert \psi_p(y) \vert^p\, dy = \lambda_{1,p}\left(0,L \right)
\end{eqnarray*}
and

\begin{eqnarray*}
\iint_\Omega \vert u_0 \vert^p\, dA &=& \iint_{\Omega_0} \vert u_0 \vert^p\, dA = \int_{I_0} \int_{g(x)}^{g(x) + L} \vert \varphi(x) \vert^p \vert \psi_p(y - g(x)) \vert^p\, dy dx \\
&=& \int_{I_0} \int_0^L \vert \varphi(x) \vert^p \vert \psi_p(y) \vert^p\, dy dx \\
&=& \int_{I_0} \vert \varphi(x) \vert^p\, dx \int_0^L \vert \psi_p(y) \vert^p\, dy = 1.
\end{eqnarray*}
Therefore, $\lambda^H_{1,p}(\Omega) = \lambda_{1,p}\left(0,L \right)$ and $u_0$ is a membrane eigenfunction for $\lambda^H_{1,p}(\Omega)$. This completes the proof.
\end{proof}

\begin{rem}\label{eigencharrem} From the proof of Theorem \ref{T1}, one can extract for $H(x,y) = |y|$ that if the membrane $\Omega$ has $C^{0,1}$ boundary and satisfies the conditions $(i)-(iii)$ and $u_p \in W^{1,p}_0(\Omega)$ is a minimizer for $\lambda_{1,p}^{H}(\Omega)$, which by homogeneity can be assumed nonnegative, then there are a $C^{0,1}$ sub-membrane $\Omega_{0} \subset \Omega$, an open subinterval $I_{0} \subset I$ and a Lipschitz function $g \colon I_{0} \to \R$ such that

\[
\Omega_0 = \left\{(x,y) \in \Omega \colon  g(x) < y < g(x) + L\right\}
\]
and $u_p(x,\cdot)$ is a positive eigenfunction of the one-dimensional $p$-laplacian for every $x \in I_0$. In particular,

\[
u_p(x,y) = \varphi(x) \psi_p(y - g(x)), \, \forall (x,y) \in \Omega_{0},
\]
where $\psi_p$ is the positive eigenfunction for the one-dimensional $p$-laplacian on $(0,L)$ and $\varphi \in W^{1,p}_{0}(I_0)$, both $L^p$-normalized. The sets $\Omega_{0}$ and $I_{0}$ can be obtained considering the maximal connected sets $\Omega'$ and $I'$ such that $u_p$ is strictly positive. In addition, Morrey's Theorem implies $\varphi \in C(\bar{I}_0)$ and the elliptic theory gives $\psi_p \in C^{1}[0,L]$. Note that $u_p$ is actually positive in $\Omega_0$ and zero on $\partial \Omega_0$. It is clear that this does not mean that $u_p$ is zero outside $\Omega_0$, just that we can ensure the existence of a subdomain such that $u_p \in W^{1,p}_{0}(\Omega_0)$.
\end{rem}

\subsection{Computation of least energy levels and the maximization conjecture}

This section is dedicated to the proof of Theorems \ref{T2} and \ref{T2.1}. Let $I$ and $\Omega_x$ be as in the previous subsection. For convenience, we next rename $L_{\frac{\pi}{2}}$ as

\[
L_\Omega = \sup \{ \vert J \vert:\,\ \text{for open intervals} \ J \subset \Omega_{x} \ \text{over all} \ x \in I\}.
\]

\begin{proof}[Proof of Theorem \ref{T2}]
Assume first that $H(x,y) = |y|$. When $\Omega$ is a rectangle whose sides are parallel to the coordinate axes, the statement $\lambda^H_{1,p}(\Omega) = \lambda_{1,p}(0,L_\Omega)$ follows directly from Theorem \ref{T1}. For an arbitrary membrane $\Omega$, note that for $u \in W_0^{1,p}(\Omega)$, since $u(x, \cdot) \in W_0^{1,p}(\Omega_x)$ for almost every $x \in I$, we have

\begin{eqnarray*}
\iint_\Omega H^p(\nabla u)\, dA &=& \int_I \int_{\Omega_x}H^p(\nabla u(x,y))\, dydx \\
&=& \int_I \int_{\Omega_x}  \vert \frac{\partial u}{\partial y}(x,y) \vert^p dydx \\
&\geq& \int_I \lambda_{1,p}(0,L_{\Omega}) \int_{\Omega_x} \vert u(x,y) \vert^p  dydx\\
&=& \lambda_{1,p}(0,L_{\Omega}) \iint_\Omega \vert u \vert^p\,  dA,
\end{eqnarray*}
so that $\lambda^H_{1,p}(\Omega) \geq \lambda_{1,p}(0,L_{\Omega})$.

On the other hand, there is a sequence of open rectangles $R_k = (a_k, b_k) \times (c_k, d_k) \subset \Omega$ such that $L_k = d_k - c_k \rightarrow L_\Omega$ as $k \rightarrow \infty$. Then, by Property \ref{PropertyP2} and the above remark, we obtain

\[
\lambda^H_{1,p}(\Omega) \leq \lambda^H_{1,p}(R_k) = \lambda_{1,p}(0,L_k),
\]
and letting $k \rightarrow \infty$, we derive

\[
\lambda^H_{1,p}(\Omega) \leq \lambda_{1,p}(0,L_\Omega).
\]
This concludes the statement for $H(x,y) = \vert y \vert$ and any membrane $\Omega$.

Now for any $H \in {\cal H}^D \setminus \{0\}$, since $H(x,y) = c\left| \cos \theta\, x + \sin \theta\, y \right|$, we have $H_A(x,y) = c \vert y \vert$ for the rotation matrix $A$ of angle $\theta - \frac{\pi}{2}$. By Property \ref{PropertyP1}, one can assume that $c = 1$. Thus, thanks to Property \ref{PropertyP4} and the first part, it follows that

\[
\lambda_{1,p}^{H}(\Omega) = \lambda_{1,p}^{H_A}(\Omega_A) = \lambda_{1,p}(0,L_{\Omega_A}).
\]
On the other hand, since $A^{T}(0,1) = (-\cos \theta, \sin \theta)$, we can describe $L_{\Omega_A}$ in an alternative way. From the invariance of Lebesgue measure under orthogonal transformations, we have

\begin{align*}
    L_{\Omega_A} &= \sup_{x \in \R}\left\{ |J| \colon J \subset (\{x\}\times\R) \cap \Omega_A \ \text {and} \ J \ \text{is connected} \right\} \\
    &= \sup_{v \in \R^2}\sup \left\{ |J| \colon J \subset \{t(0,1) + v \colon t \in \R\} \cap \Omega_A \ \text{and} \ J  \ \text{is connected}\right\} \\
    &= \sup_{v \in \R^2}\sup \left\{ |J| \colon A^{T}J \subset A^{T} \{t(0,1) + v \colon t \in \R\} \cap \Omega \ \text{and} \ J  \ \text{is connected}\right\} \\
    &= \sup_{v \in \R^2}\sup \left\{ |A^{T}J| \colon A^{T}J \subset \{tA^{T}(0,1) + A^{T}v \colon t \in \R\} \cap \Omega \ \text{and} \ A^{T}J  \ \text{is connected}\right\} \\
    &=\sup_{v \in \R^2}\sup \left\{ |J| \colon J \subset \{t(-\cos{\theta},\sin{\theta}) + Av \colon t \in \R\}\cap \Omega \ \text{and} \ J  \ \text{is connected}\right\}  \\
    &= \sup_{v \in \R^2}\sup \left\{ |J| \colon J \subset \{t(-\cos{\theta},\sin{\theta}) + v \colon t \in \R\}\cap \Omega \ \text{and} \ J  \ \text{is connected}\right\}  \\
    &= L_\theta.
\end{align*}
Therefore,

\[
\lambda_{1,p}^{H}(\Omega) = \lambda_{1,p}(0,L_\theta)
\]
and this ends the proof.
\end{proof}

As an application of this result, we give a very short proof of the maximization conjecture.

\begin{proof}[Proof of Theorem \ref{T2.1}]
Let any $H \in {\cal H} \setminus \{0\}$. As already argued in the proof of Proposition \ref{P.1}, there always exist a rotation matrix $A$ and a constant $c > 0$ such that $H_A(x,y) \geq H_0(x,y) := c \vert y \vert$. Consider for each integer $k \geq 1$ the open rectangle $\Omega_k = A(R_k)$ where $R_k = (0, k) \times (0, 1/k)$. Clearly, $\vert \Omega_k \vert = 1$ and, by Properties \ref{PropertyP3} and \ref{PropertyP4} and Theorem \ref{T2}, we have

\[
\lambda_{1,p}^H(\Omega_k) = \lambda_{1,p}^{H_A}(R_k) \geq \lambda_{1,p}^{H_0}(R_k) = c^p \lambda_{1,p}(0, 1/k).
\]
Then, letting $k \rightarrow \infty$ in this inequality, we deduce that $\lambda_{1,p}^H(\Omega_k) \rightarrow \infty$ and the proof follows.
\end{proof}

\section{Solution of the isoanisotropic problems}

This section is devoted to the complete proof of Theorems \ref{T3} and \ref{T4}. Particularly, Remark \ref{eigencharrem} related to Theorem \ref{T1} and Theorem \ref{T2} play a fundamental role in the latter.

\subsection{Anisotropic rigidity of maxima optimization}

\begin{proof}[Proof of Theorem \ref{T3}]

The computation of the optimal upper anisotropic constant $\lambda_{1,p}^{\max}(\Omega)$ is immediate. In effect, for any $H \in \mathbb{S}({\cal H})$, we have

\begin{eqnarray*}
\lambda^H_{1,p}(\Omega) &=& \inf \left\{ \iint_\Omega H^p(\nabla u)\, dA:\ u \in W^{1,p}_0(\Omega),\ \Vert u \Vert_{L^p(\Omega)} = 1 \right\}\\
&\leq& \inf \left\{ \iint_\Omega \vert \nabla u \vert^p\, dA:\ u \in W^{1,p}_0(\Omega),\ \Vert u \Vert_{L^p(\Omega)} = 1 \right\}\\
&=& \lambda^{\bar{H}}_{1,p}(\Omega)
\end{eqnarray*}
since $H \leq \bar{H}$, where $\bar{H}(x,y) = \vert (x,y) \vert$ denotes the Euclidean norm. This fact along with $\bar{H} \in \s({\cal H})$ yield $\lambda_{1,p}^{\max}(\Omega) = \lambda^{\bar{H}}_{1,p}(\Omega) = \lambda_{1,p}(\Omega)$.

We next establish the rigidity of the above equality provided that the boundary of $\Omega$ is $C^{1,\alpha}$. It suffices to show that

\[
\lambda^H_{1,p}(\Omega) < \lambda^{\bar{H}}_{1,p}(\Omega)
\]
for every $H \in \mathbb{S}({\cal H})$ with $H \neq \bar{H}$.

For such an anisotropy $H$, we have $H(\cos \theta, \sin \theta) < 1$ for some $\theta \in [0,\pi]$. Consider the family of half planes $E_\mu = \{(x,y) \in \R^2:\ \cos \theta\, x + \sin \theta\, y > \mu\}$ for $\mu \in \R$ and define

\[
\mu_0 = \inf\{\mu \in \R:\ E_\mu \cap \Omega \neq \emptyset\}.
\]
Since $\Omega$ is bounded, we have that $\mu_0$ is finite and $\partial E_{\mu_0} \cap \partial \Omega \neq \emptyset$.

Let now $u_p \in W^{1,p}_0(\Omega)$ be a positive eigenfunction for $\lambda_{1,p}(\Omega)$ with $\Vert u \Vert_p = 1$. From the elliptic regularity theory of the $p$-laplace operator (see e.g. \cite{To1}), it is well known that $u_p \in C^1(\overline \Omega)$ since $\partial \Omega \in C^{1,\alpha}$. Moreover, by Hopf's Lemma, we have $\frac{\partial u_p}{\partial \nu} < 0$ on $\partial \Omega$, where $\nu$ denotes the outward unit normal field to $\Omega$. Hence, $\nabla u_p$ is an exterior nonzero normal field on $\partial \Omega$. Then, for $(x_0, y_0) \in \partial E_{\mu_0} \cap \partial \Omega$, we have $\nabla u_p(x_0,y_0) = \vert \nabla u_p(x_0,y_0) \vert (\cos \theta, \sin \theta)$. Therefore, $H(\nabla u_p(x_0, y_0)) = \vert \nabla u_p(x_0,y_0) \vert H(\cos \theta, \sin \theta) < \vert \nabla u_p(x_0, y_0) \vert$ and so $H(\nabla u_p) <  \vert \nabla u_p \vert$ in the set $\Omega_\delta = \Omega \cap B_\delta(x_0,y_0)$ for some $\delta > 0$ small enough. Using this strict inequality in $\Omega_\delta$ and $H(\nabla u_p) \leq  \vert \nabla u_p \vert$ in $\Omega \setminus \Omega_\delta$, we get

\begin{eqnarray*}
\lambda^H_{1,p}(\Omega) &\leq& \iint_\Omega H^p(\nabla u_p)\, dA = \iint_{\Omega_\delta} H^p(\nabla u_p)\, dA + \iint_{\Omega \setminus \Omega_\delta} H^p(\nabla u_p)\, dA \\
&<& \iint_\Omega \vert \nabla u_p \vert^p\, dA = \lambda_{1,p}(\Omega).
\end{eqnarray*}
\end{proof}

\subsection{Characterization of anisotropic extremizers for minima optimization}

\begin{propo}\label{propbound}
For any $H \in {\cal H}$, there is $H_0 \in {\cal H}^{D}$ such that $H \geq H_{0}$ and $\lVert H \rVert = \lVert H_{0} \rVert$.
\end{propo}

\begin{proof}
The case where $H \in {\cal H}^D$ is trivial because the result follows for $H = H_0$. For $H \in {\cal H}^P$, by Proposition \ref{P1}, it suffices to show that there is $\theta \in [0,\pi]$ such that

\[
H(x,y) \geq \Vert H \Vert\, \left|x\cos{\theta}+y\sin{\theta} \right| =: H_0(x,y)
\]
for all $(x,y) \in \R^2$. Actually, since $H$ is a norm on $\R^2$, the set $C = H^{-1}\left(-\infty,\lVert H \rVert\right]$ is a convex body. Let $\theta_{0} \in [0,\pi]$ be such that $H(\cos{\theta_{0}},\sin{\theta_{0}}) = \Vert H \Vert$ and take the line $\mathbb{L}$ tangent to $C$ that passes through $(\cos{\theta_{0}},\sin{\theta_{0}})$. It is clear that there is $\theta \in [0,\pi]$ such that one can write this line as

\[
\mathbb{L} =  \{(x,y) \in \R^2 :\, x\cos{\theta}+y\sin{\theta} = \cos{\theta_{0}}\cos{\theta}+\sin{\theta_{0}}\sin{\theta} = \cos{(\theta - \theta_{0})}\}.
\]
Notice that

\[
x\cos{\theta}+y\sin{\theta} < \cos{(\theta - \theta_{0})}, \ \forall (x,y) \in \interior(C).
\]
For any $(x,y) \in \R^2$ such that $x\cos{\theta}+y\sin{\theta} \neq 0$, we see that the vector

\[
(z,w) = \frac{\cos{(\theta - \theta_{0})}}{x\cos{\theta}+y\sin{\theta}}(x,y)
\]
is in the complement of $\interior(C)$ because

\[
z \cos{\theta}+ w \sin{\theta} = \frac{\cos{(\theta - \theta_{0})}}{x\cos{\theta}+y\sin{\theta}}\left(x\cos{\theta}+y\sin{\theta}\right) = \cos{(\theta - \theta_{0})},
\]
and thus $H(z,w) \in [\Vert H \Vert, +\infty)$. In other words,

\[
H\left(\frac{\cos{(\theta - \theta_{0})}}{x\cos{\theta}+y\sin{\theta}}(x,y)\right) \geq \Vert H \Vert
\]
and so using the fact that $0 < |\cos{(\theta - \theta_{0})}| \leq 1$, we get

\[
H(x,y) \geq \frac{|x\cos{\theta}+y\sin{\theta}|}{|\cos{(\theta - \theta_{0})}|}\, \Vert H \Vert \geq \Vert H \Vert\, \left|x\cos{\theta}+y\sin{\theta}\right|.
\]
If $x \cos{\theta} + y \sin{\theta} = 0$, then the nonnegativity of $H$ yields

\[
H(x,y) \geq 0 = \Vert H \Vert\, \left|x\cos{\theta}+y\sin{\theta}\right|.
\]
In any case, we derive the desired inequality.
\end{proof}

For the remainder of this section, it is convenient to consider the sphere of ${\cal H}$ restricted to ${\cal H}^P$ and ${\cal H}^D$, denoted respectively by

\[
\mathbb{S}({\cal H}^P) = \mathbb{S}({\cal H})\cap{\cal H}^P = \left\{H \in {\cal H}^P \colon \lVert H \rVert = 1\right\},
\]

\[
\mathbb{S}({\cal H}^D) = \mathbb{S}({\cal H})\cap{\cal H}^D = \left\{H \in {\cal H}^D \colon \lVert H \rVert = 1\right\}.
\]

\begin{propo}\label{T3degenerate} For any membrane $\Omega \subset \R^2$ and $p > 1$, we have

\[
\lambda_{1,p}^{\min}(\Omega) = \inf_{H \in \mathbb{S}({\cal H}^D)} \lambda_{1,p}^{H}(\Omega) = \inf_{\theta \in [0,\pi]} \lambda_{1,p}(0,L_{\theta}).
\]
Moreover, the infimum is attained if and only if $\Omega$ has optimal anisotropic design.
\end{propo}

\begin{proof} The two equalities follow readily from the definition of $\lambda_{1,p}^{\min}(\Omega)$ and from Propositions \ref{P1} and \ref{propbound} and Theorem \ref{T2}. On the other hand, the last infimum is attained on $[0, \pi]$ if and only if the width function $L: [0, \pi] \rightarrow \R$ has a global maximum point, but this just means that $\Omega$ has optimal anisotropic design.
\end{proof}

\begin{propo}\label{T3degenerate1}
Let $\Omega$ be a $C^{0,1}$ membrane and let $p > 1$. For any $H \in \mathbb{S}({\cal H}^P)$ and $H_0 \in \mathbb{S}({\cal H}^D)$ such that $H \geq H_0$, it holds that $\lambda_{1,p}^{H}(\Omega) > \lambda_{1,p}^{H_0}(\Omega)$. In particular, $\lambda_{1,p}^{\min}(\Omega)$ has at most anisotropic extremizers in $\mathbb{S}({\cal H}^D)$.
\end{propo}

\begin{proof} Let $H \in \mathbb{S}({\cal H}^P)$ and $H_0 \in \mathbb{S}({\cal H}^D)$ be as in the statement and assume by contradiction that $\lambda_{1,p}^{H}(\Omega) = \lambda_{1,p}^{H_0}(\Omega)$. By Proposition \ref{P1} and Property \ref{PropertyP4}, it suffices to suppose that $H_0(x,y) = |y|$. Since $H \in {\cal H}^P$, there is a minimizer $u_p \in W^{1,p}_{0}(\Omega)$ for $\lambda_{1,p}^{H}(\Omega)$. From the inequalities

\[
\lambda_{1,p}^{H_0}(\Omega) = \lambda_{1,p}^{H}(\Omega) = \iint_{\Omega} H^{p}(\nabla u_p)\, dA \geq  \iint_{\Omega} H_{0}^{p}(\nabla u_p)\, dA \geq \lambda_{1,p}^{H_0}(\Omega),
\]
it follows that $u_p$ is a minimizer for $\lambda_{1,p}^{H_0}(\Omega)$ as well. Thus, by Remark \ref{eigencharrem}, there is a number $L > 0$, an interval $I_0 = (a,b)$, a Lipschitz function $g \colon I_0 \to \R$ (since $\partial \Omega \in C^{0,1}$), a sub-membrane
\[
\Omega_0 = \left\{(x,y) \in \Omega \colon  g(x) < y < g(x) + L\right\}
\]
and a function $\varphi \in W_0^{1,p}(I_0)$ with $\Vert \varphi \Vert_p = 1$ where $u_p(x,y) = \varphi(x)\psi_{p}(y - g(x))$ is positive in the interior of $\Omega_0$ and zero on the boundary of $\Omega_0$. Here, $\psi_p$ denotes the positive eigenfunction for the one-dimensional $p$-laplacian on $(0,L)$ with $\Vert \psi_p \Vert_p = 1$. Therefore,

\[
D_y u_p:= \frac{\partial u_p}{\partial y} \in C(\Omega_{0})\setminus\{0\}
\]
and it is zero only on the curve $\left\{ (x,y) \in \Omega \colon y-g(x) = L/2\right\}$. Hence, $D_y u_p$ is nonzero almost everywhere in $\Omega_{0}$.

Since $H$ is continuous and $H(x,0) > 0 = 2H_{0}(x,0)$ for all $x \neq 0$, then

\[
U := \left(H-2H_0\right)^{-1}(0,+\infty)
\]
is a nonempty open set. Consider now the set

\[
W := (\nabla u_p)^{-1} (U).
\]
We claim that $W$ has positive Lebesgue measure. To check this, first define for any $x \neq 0$:

\[
\varepsilon(x) := \sup \left\{y>0 \colon (x,y) \in U\right\}.
\]
Clearly, this function satisfies $\varepsilon(-x) = \varepsilon(x)$ and $\{x\} \times (-\varepsilon(x), \varepsilon(x)) \subset U$. Besides, thanks to the homogeneity of $H - 2H_0$, one has $\varepsilon(tx) = t\varepsilon(x)$ for every $t>0$ and $x \neq 0$ and so $\varepsilon(x_2) \geq \varepsilon(x_1)$ whenever $|x_2| \geq |x_1|$.

Now see that, since there is $(x_0,y_0) \in \Omega_0$ such that $D_y u_p (x_0,y_0) = 0$ and due to continuity of $D_y u_p$, the open set $A_{\eta} = (D_y u_p)^{-1}(-\eta,\eta)$ is nonempty in $\Omega_0$ for every $\eta > 0$.

Set $D_x u_p:= \frac{\partial u_p}{\partial x}$ and define $B_{\delta} = \{(x,y) \in \Omega_{0} \colon | D_x u_p(x,y)| > \delta\}$ for each $\delta > 0$. Using the fact that $[-\delta,\delta]^{c} \times (-\varepsilon(\delta),\varepsilon(\delta)) \subset U$, we have
\[
W = (\nabla u_p)^{-1} (U) \supset (\nabla u_p)^{-1}\left([-\delta,\delta]^{c} \times (-\varepsilon(\delta),\varepsilon(\delta))\right) \supset B_{\delta} \cap A_{\varepsilon(\delta)}.
\]
Then, it suffices to prove that $B_{\delta} \cap A_{\varepsilon(\delta)}$ has positive Lebesgue measure for some $\delta > 0$.

Assume by contradiction that $|B_{\delta} \cap A_{\varepsilon(\delta)}| = 0$ for every $\delta > 0$. Setting $f(t) = (t,g(t)+y-g(x))$ and using that $g$ is Lipschitz, for any point $(x,y) \in \Omega_{0}$, we have

\begin{equation} \label{truc}
u_{p}(x,y) = u_{p}(x,g(x)+y-g(x)) = \int_{a}^{x} (u_p \circ f)'(t) \, dt = \int_{a}^{x} (D_x u_p)(f(t)) + (D_y u_p)(f(t)) g'(t) \, dt.
\end{equation}
Now define $\omega(\delta)$ as the positive number such that

\[
\left(\psi_{p}^\prime\right)^{-1}(-\varepsilon(\delta)\, \lVert \varphi \rVert_{\infty}^{-1},\, \varepsilon(\delta)\, \lVert \varphi \rVert_{\infty}^{-1}) = \left(L/2 -\omega(\delta),L/2+\omega(\delta)\right).
\]
Note that both $\varepsilon(\delta)$ and $\omega(\delta)$ converge to $0$ as $\delta \rightarrow 0$. From the equality \eqref{truc}, we get

\begin{align*}
\int_{a}^{b}\int_{g(x)+L/2-\omega(\delta)}^{g(x)+L/2+\omega(\delta)} |u_p (x,y)| \, dy \, dx & = \int_{a}^{b}\int_{g(x)+L/2-\omega(\delta)}^{g(x)+L/2+\omega(\delta)} \left|\int_{a}^{x} (D_x u_p)(f(t)) + (D_y u_p)(f(t)) g'(t) \ dt\right| \, dy \, dx \\
& \leq \int_{a}^{b}\int_{g(x)+L/2-\omega(\delta)}^{g(x)+L/2+\omega(\delta)}\int_{a}^{x} |D_x u_p(f(t))| + |D_y u_p(f(t))||g'(t)| \, dt \, dy \, dx \\
& \leq \int_{a}^{b}\int_{g(x)+L/2-\omega(\delta)}^{g(x)+L/2+\omega(\delta)}\int_{a}^{b} |D_x u_p(f(t))| + |D_y u_p(f(t))||g'(t)| \, dt \, dy \, dx.
\end{align*}
On the other hand, from the definition of $A_{\varepsilon(\delta)}$, we have

\[
\{(x,y) \in \R^2 \colon\ a < x < b,\ g(x)+L/2-\omega(\delta) < y < g(x)+L/2+\omega(\delta)\} \subset A_{\varepsilon(\delta)}
\]
and therefore $f(t) \in A_{\varepsilon(\delta)}$ for every $t \in (a,b)$. Then, since $|B_{\delta} \cap A_{\varepsilon(\delta)}| = 0$, the previous estimate yields

\begin{align*}
  \int_{a}^{b}\int_{g(x)+L/2-\omega(\delta)}^{g(x)+L/2+\omega(\delta)} |u_p (x,y)| \, dy \, dx & \leq \int_{a}^{b}\int_{g(x)+L/2-\omega(\delta)}^{g(x)+L/2+\omega(\delta)}\int_{a}^{b} \delta + \varepsilon(\delta)|g'(t)| \, dt \, dy \, dx \\
  & = \int_{a}^{b}\int_{g(x)+L/2-\omega(\delta)}^{g(x)+L/2+\omega(\delta)} \delta (b-a) + \varepsilon(\delta)\lVert g' \rVert_{1} \, dy \, dx \\
  & = (b-a)2\omega(\delta)\left(\delta (b-a) + \varepsilon(\delta)\lVert g' \rVert_{1}\right).
\end{align*}
Dividing both sides by $2\omega(\delta)$, we have

\[
\int_{a}^{b} \frac{1}{2\omega(\delta)}\int_{g(x)+L/2-\omega(\delta)}^{g(x)+L/2+\omega(\delta)} |u_p (x,y)| \, dy \, dx \leq (b-a)\left(\delta (b-a) + \varepsilon(\delta)\lVert g' \rVert_{1}\right),
\]
and letting $\delta \rightarrow 0$, one obtains

\[
\int_{a}^{b} |u_{p}(x,g(x)+L/2)| \, dx = 0,
\]
which gives a contradiction since the function $u_{p}(x,g(x)+L/2)$ is positive for every $x \in (a,b)$. Consequently, there is $\delta > 0$ such that $|B_{\delta} \cap A_{\varepsilon(\delta)}| > 0$ and thus $W$ has positive measure.

Finally, with the aid of this last fact, we derive the contradiction

\begin{align*}
    \lambda_{1,p}^{H}(\Omega) &= \iint_{\Omega} H^{p}(\nabla u_p)\, dA = \iint_{\Omega\setminus W} H^{p}(\nabla u_p)\, dA + \iint_{W} H^{p}(\nabla u_p)\, dA \\  &\geq \iint_{\Omega\setminus W} H_{0}^{p}(\nabla u_p)\, dA + \iint_{W} 2^pH_{0}^{p}(\nabla u_p)\, dA \\
    &= \iint_{\Omega} H_{0}^{p}(\nabla u_p)\, dA + \iint_{W} (2^p-1)H_{0}^{p}(\nabla u_p)\, dA > \lambda_{1,p}^{H_{0}}(\Omega).
\end{align*}
and the proof of the proposition follows.
\end{proof}

As byproduct of the above developments, we have

\begin{proof}[Proof of Theorem \ref{T4}.] It follows from the combination between Propositions \ref{T3degenerate} and \ref{T3degenerate1}.
\end{proof}

\subsection{Proof of two shape inequalities}

In this last subsection we use the strength of Theorem \ref{T4} to give a brief proof of the shape inequalities \eqref{ID-min} and \eqref{IP-min}.

\begin{proof}[Proof of Theorem \ref{T5}]
By scaling property of $\lambda_{1,p}^{\min}(\Omega)$ with respect to $\Omega$, for the proof of the inequalities $(i)$ and $(ii)$ suffices to consider the respective normalized cases, namely, $\rm{diam}(\Omega) = 1$ and $\vert \Omega \vert = 1$.

By Theorem \ref{T4}, for any membrane $\Omega$, we know that

\[
\lambda_{1,p}^{\min}(\Omega) = \inf_{\theta \in [0,\pi]} \lambda_{1,p}(0, L_\theta) = \lambda_{1,p}(0,1) [\sup_{\theta \in [0, \pi]} L_\theta]^{-p}.
\]

For the proof of $(i)$, assume that $\Omega$ has unit diameter. Since $\sup_{\theta \in [0, \pi]} L_\theta \leq \rm{diam}(\Omega) = 1$, we have

\[
\lambda_{1,p}^{\min}(\Omega) \geq \lambda_{1,p}(0,1).
\]
Moreover, the equality holds if and only if $\sup_{\theta \in [0, \pi]} L_\theta = 1 = \rm{diam}(\Omega)$.

For the proof of $(ii)$, consider a convex membrane $\Omega$ with unit area. From the convexity, there is $\theta_0 \in [0, \pi]$ such that $L_{\theta_0} = \rm{diam}(\Omega)$. Therefore, using the above formula for $\lambda_{1,p}^{\min}(\Omega)$, we get

\[
\lambda_{1,p}^{\min}(\Omega) = \lambda_{1,p}(0, L_{\theta_0}) = \lambda_{1,p}(0, \rm{diam}(\Omega)).
\]
On the other hand, the classical isodiametric inequality gives us

\[
\rm{diam}(\Omega) \geq \frac{2}{\sqrt{\pi}},
\]
and thus

\[
\lambda_{1,p}^{\min}(\Omega) \leq \lambda_{1,p}(0, 2/\sqrt{\pi}).
\]
Moreover, equality holds if and only if the isodiametric inequality becomes equality, which in turn only occurs for disks.
\end{proof}

\vs{0.5cm}

\n {\bf Acknowledgments:} The first author was partially supported by CAPES (PROEX 88887.712161/2022-00) and the second author was partially supported by CNPq (PQ 307432/2023-8, Universal 404374/2023-9) and Fapemig (Universal APQ 01598-23).\\

\n {\bf Data availability:} Data sharing not applicable to this article as no datasets were generated or analyzed during the current study.\\

\n {\bf Conflict of interest:} The authors declare that there are no financial, competing or conflict of interest.

\end{document}